\documentclass[leqno]{amsart} 
\usepackage{amsfonts,amssymb,amsmath,amsgen,amsthm} 
\usepackage{ucs}
\usepackage{enumerate}
\usepackage{graphicx}

\usepackage{relsize} 
\usepackage{bigints}

\usepackage{hyperref}

\theoremstyle{plain}
\newtheorem{theorem}{Theorem}[section]
\newtheorem{lemma}[theorem]{Lemma} 
\newtheorem{corollary}[theorem]{Corollary}
\newtheorem{proposition}[theorem]{Proposition}
\newtheorem{remark}[theorem]{Remark}
\newtheorem{definition}[theorem]{Definition}

\def\R{{\mathbb R}}
\def\T{{\mathbb T}}
\def\N{{\mathbb N}}
\def\Z{{\mathbb Z}}

\def\P{\mathbf P}
\def\Q{\mathbf Q}
\def\dd{\mathrm d}

\def\({\left(}
\def\){\right)}
\def\<{\left\langle}
\def\>{\right\rangle}

\def\1{{\mathbf 1}}

\def\eps{\varepsilon}

\newcommand{\D}{\mathbf{D}}
\newcommand{\A}{\mathbf{A}}

\newcommand{\seps}{\sigma_{\eps}}
\newcommand{\reps}{\rho_{\eps}}

\newcommand{\Ueps}{U_{\eps}}
\newcommand{\ueps}{u_{\eps}}

\newcommand{\meps}{m_{\eps}}
\newcommand{\dive}{\mathop{\mathrm{div}}}
\newcommand{\Feps}{F_{\eps}}

\newcommand{\Heps}{H_{\eps}}
\newcommand{\eith}{e^{it\Heps}}

\newcommand{\Seps}{\mathbf{S_{\eps}}}
\newcommand{\Teps}{\mathbf{T_{\eps}}}

\newcommand{\Kd}{\mathbf{K_{\delta}}}
\newcommand{\Sd}{\mathbf{S_{\delta}}}
\newcommand{\Sb}{\mathbf{S}}
\newcommand{\Ab}{\mathbf{A}}
\newcommand{\rd}{\rho_{\delta}}
\newcommand{\ud}{u_{\delta}}
\newcommand{\fd}{f_{\delta}}
\newcommand{\hd}{h_{\delta}}
\newcommand{\gd}{g_{\delta}}
\newcommand{\vd}{v_{\delta}}
\newcommand{\phid}{\phi_{\delta}}

\DeclareMathOperator{\diver}{div}

\DeclareMathOperator{\Hess}{Hess}

\numberwithin{equation}{section}

\date\today

\title{On the low Mach number limit for Quantum Navier-Stokes equations}

\author[P. Antonelli]{Paolo Antonelli}
\address{Gran Sasso Science Institute, viale Francesco Crispi, 7, 67100 L'Aquila}
\email{paolo.antonelli@gssi.it}

\author[L.E. Hientzsch]{Lars Eric Hientzsch}
\address{Gran Sasso Science Institute, viale Francesco Crispi, 7, 67100 L'Aquila}
\email{larseric.hientzsch@gssi.it}

\author[P. Marcati]{Pierangelo Marcati}
\address{Gran Sasso Science Institute, viale Francesco Crispi, 7, 67100 L'Aquila}
\email{pierangelo.marcati@gssi.it}

\begin{document}

\begin{abstract}
We investigate the low Mach number limit for the 3--$D$ quantum Navier-Stokes system. For general ill-prepared initial data, we prove strong convergence of finite energy weak solutions to weak solutions of the incompressible Navier Stokes equations. 
Our approach relies on a quite accurate dispersive analysis for the acoustic part, governed by the well known Bogoliubov dispersion relation for the elementary excitations of the weakly-interacting Bose gas. {Once we have a control of the acoustic dispersion, the a priori bounds provided by the energy and Bresch-Desjardins entropy type estimates lead} to the strong convergence. Moreover, for well-prepared data we show that the limit is a Leray weak solution, namely it satisfies the energy inequality. Solutions under consideration in this paper are not smooth enough to allow for the use of relative entropy techniques. 
\end{abstract}

\maketitle

\section{Introduction}
In this paper, we study the low Mach number limit for a dispersive-diffusive fluid model, usually called  the Quantum-Navier-Stokes equations (QNS). The system is posed on $(0,T)\times \R^3$, 
\begin{equation}\label{eq:QNS1}
\begin{aligned}
\begin{cases}
&\partial_t\rho+\dive (\rho u)=0,\\
&\partial_t(\rho u)+\dive\left(\rho u \otimes u \right)+\nabla P(\rho)=2\nu \dive(\rho\D u)+2\kappa^2\rho \nabla\left(\frac{\Delta\sqrt{\rho}}{\sqrt{\rho}}\right),
\end{cases}
\end{aligned}
\end{equation}
the unknowns are given by the mass density $\rho$ and the fluid velocity field $u$. We consider a pressure given by the usual $\gamma$-law, i.e. $P(\rho)=\frac{1}{\gamma}\rho^{\gamma}$, with $\gamma>1$, and where $\nu,\kappa>0$ denote the viscosity and capillarity coefficients respectively. 
The energy we consider for this system \eqref{eq:QNS} is given by
\begin{equation}\label{eq:en}
E(t)=\int_{\R^3}\frac{1}{2}\rho|u|^2+2\kappa^2|\nabla\sqrt{\rho}|^2+\pi(\rho) \dd x,
\end{equation}
where the internal energy takes the form
\begin{equation}\label{eq:ren_ie1}
\pi=\pi(\rho)=\frac{\rho^{\gamma}-1-\gamma(\rho-1)}{\gamma(\gamma-1)}.
\end{equation}
Thus, the finite energy assumption yields
\begin{equation}\label{eq:farfield}
\rho \rightarrow 1 \hspace{1cm} \text{as} \hspace{0.5cm} |x|\rightarrow \infty.
\end{equation}

System \eqref{eq:QNS1} enters the more general class of Navier-Stokes-Korteweg systems and it is sometimes used as a model for dissipative quantum fluids particularly for numerical purposes, namely
\begin{equation*}
\begin{aligned}
\begin{cases}
&\partial_t\rho+\diver(\rho u)=0,\\
&\partial_t(\rho u)+\diver(\rho u \otimes u)+\nabla P(\rho)=2\nu\diver(\mathbb{S})+ \kappa^2 \diver(\mathbb{K}),
\end{cases}
\end{aligned}
\end{equation*}
where the viscous stress tensor $\mathbb{S}=\mathbb{S}(\nabla u)$ is given by
\begin{equation*}
\mathbb{S}=h(\rho)\D u+g(\rho)\diver u \mathbb{I},
\end{equation*}
while the capillary (dispersive) term $\mathbb{K}=\mathbb{K}(\rho,\nabla\rho)$ reads
\begin{equation*}
\mathbb{K}=\left(\rho\diver(k(\rho)\nabla\rho)-\frac{1}{2}(\rho k'(\rho)-k(\rho)|\nabla\rho|^2)\right)\mathbb{I}-k(\rho)\nabla\rho\otimes\nabla\rho.
\end{equation*}
System \eqref{eq:QNS1} is then recovered by choosing $h(\rho)=\rho$, $g(\rho)=0$ and $k(\rho)=\frac{1}{\rho}$.
The QNS equations can also be derived from a Chapman-Enskog expansion for the Wigner equation with a BGK term \cite{BM, JMi}, see also \cite{J} where several dissipative quantum fluid models are derived by means of a moment closure of (quantum) kinetic equations with appropriate choices of the collision terms.
A more robust model, from the point of view of quantum physics, is given by the inviscid counterpart of \eqref{eq:QNS}, namely the classical Quantum Hydrodynamic system (QHD), see \cite{AM1,AM2,AM16, AHMZ}, that arises in many different situations, for instance  as a hydrodynamic model in superfluids \cite{Khal} and Bose-Einstein condensates \cite{PS} or to describe the carrier transport in semiconductors \cite{Gar}.

After a suitable rescaling (see subsection \ref{subsec:scaling}), the system \eqref{eq:QNS1} reads, 
\begin{equation}\label{eq:QNS}
\begin{aligned}
\begin{cases}
&\partial_t\rho_{\eps}+\dive (\rho_{\eps}u_{\eps})=0,\\
&\partial_t(\rho_{\eps}u_{\eps})+\dive\left(\rho_{\eps}u_{\eps} \otimes u_{\eps}\right)+\frac{1}{\eps^2}\nabla P(\rho_{\eps})=2\nu \dive(\rho_{\eps}\D u_{\eps})+2\kappa^2\rho_{\eps}\nabla\left(\frac{\Delta\sqrt{\rho_{\eps}}}{\sqrt{\rho_{\eps}}}\right),
\end{cases}
\end{aligned}
\end{equation}
with initial data 
\begin{equation*}
\begin{aligned}
\reps(0,x)&=\rho_{\eps}^0,\\
(\reps\ueps)(0,x)&=\rho_{\eps}^0u_{\eps}^0,
\end{aligned}
\end{equation*}
where $\eps\ll 1$ is the scaled Mach number. Therefore, the scaled internal energy becomes
\begin{equation}\label{eq:ren_ie}
   \pi_{\eps}=\pi_{\eps}(\rho_{\eps})=\frac{\rho^{\gamma}-1-\gamma(\rho-1)}{\eps^2\gamma(\gamma-1)}.
\end{equation}
In the low Mach number regime, i.e. in the limit as $\eps\to0$, the dynamics of \eqref{eq:QNS} is formally governed by the incompressible Navier-Stokes equations, 
\begin{equation}\label{eq:INS}
\partial_t u +u\cdot \nabla u+\nabla p=\nu \Delta u, \qquad \dive u=0.
\end{equation}
The aim of this paper is to rigorously study this limit in its full generality, i.e. by considering arbitrary finite energy initial data without imposing further regularity or smallness assumptions and in particular without being well-prepared.  
This class of initial data cannot provide in the limit smooth solutions to incompressible Navier-Stokes, for this reason we have to exclude the use of relative entropy methods \cite{DM16}, \cite{FN}.

The QNS system entails some mathematical difficulties due to the possible appearance of vacuum regions. Indeed, the degenerate viscosity prevents a suitable control of the velocity field in the vacuum. In particular this yields some problems in establishing the necessary compactness estimates on the convective term $\rho_\eps u_\eps\otimes u_\eps$. 
As a consequence, the mathematical analysis of fluid dynamical equations with degenerate viscosities differs substantially from the standard Feireisl-Lions theory \cite{L96, F}. On the other hand, for certain compressible fluids with degenerate viscosity a further entropy estimate is available, first introduced by Bresch and Desjardins in \cite{BDL} and later extended to a large class of Navier-Stokes-Korteweg systems \cite{BD}. For system \eqref{eq:QNS1}, these estimates yield a regularizing effect for the mass density \cite{JqNS, AS15}, hence in our case they also allow for a better control of the far-field behavior compared to the classical compressible Navier-Stokes equations \cite{LM98}, see Lemmata \ref{lem:apriori} and \ref{lem:fluctuations}.
Finite energy weak solutions to \eqref{eq:QNS} in the two and three dimensional torus were studied in \cite{AS, LV16}.
The Cauchy Problem on $\R^d$, $d=2,3$, with non-trivial far-field behavior has recently been addressed in \cite{AHS20}. The existence theory of finite energy weak solutions in \cite{AHS20} is based on the standard argument of invading domains, suitable truncations, and the available theory for $\T^d$ of \cite{LV16}. Besides the mathematical difficulties originated from the degenerate viscosity, in this framework we also have to cope with the lack of integrability of the mass density due to non-trivial boundary conditions. 

One of the main tools in this paper is provided by a class of suitable Strichartz estimates, that allow to capture more accurately the different dispersive scales involved in the propagation of the acoustic waves,  as a consequence of the specific dispersion relation. 
Indeed, contrarily to the classical case where the fluctuations evolve accordingly to the classical wave equation \cite{LM98, DG99, U}, here in our problem the presence of the quantum term contributes in a non-trivial way to the dispersion relation, especially at high frequencies. 
The dispersion relation inferred here, see formula \eqref{eq:Bog_sp} below, is strictly related to the Bogoliubov spectrum describing excitations in a Bose-Einstein condensate, which predicts the superfluid behavior of the gas \cite{B, BBCS, S}. 
This is somehow reminiscent of the analysis of fluctuations done when studying the quasi-neutral limit for a class of Navier-Stokes-Korteweg systems \cite{DM08,DM15}.
The analysis related to the dispersion relation \eqref{eq:Bog_sp} can be regarded as the $\eps-$version of the results in \cite{GNT05}. We will present this analysis in the Appendix \ref{app:Strichartz}. 
Here we remark that since the dispersion relation \eqref{eq:Bog_sp} is not homogeneous, we cannot obtain our estimates by a rescaling argument and we need to adapt the proof in \cite{GNT05}. 
On the other hand, it should be remarked that if we perform a frequency splitting as in \cite{BDS10}, then the estimates \eqref{eq:Strichartz1}, \eqref{eq:Strichartz2} deteriorate at low frequencies. A more detailed explanation can be found in Section \ref{sec:acw} and Appendix \ref{app:Strichartz}.

Furthermore, in the limit we recover a weak solution of the incompressible Navier-Stokes equation  $u\in L^{\infty}(0,T;L^2(\R^3))\cap L^2(0,T,\dot{H}^1(\R^3))$.  We remark that we are able to obtain bounds on the gradient of the limiting solution to \eqref{eq:INS}, even though at fixed $\eps>0$ only a weak version of the energy inequality is available 
\cite{AS, AS18, LV16}, see also the discussion in Section \ref{sec:pre}. However, this weak version of the energy inequality will anyway yield the aforementioned natural bounds on the gradient of the velocity field in the low Mach number limit.
In fact, thanks to some uniform bounds satisfied by the momentum density, we can also infer further smoothing properties for the limiting solution to \eqref{eq:INS}, see Theorem \ref{thm:main} and Proposition \ref{prop:boundu} for more details.
Due to the presence of an initial layer which cannot be avoided for general ill-prepared data, the weak solution enters the Leray class only if further assumptions on the initial data are made. More precisely, only for well-prepared data it is possible to show that the solutions obtained in the limiting procedure satisfy the energy inequality. 

The study of singular limits for fluid dynamical equations occupies a vast portion of mathematical literature, for a more comprehensive introduction to the topic we address the reader to the monograph \cite{FN} and the reviews \cite{A, M}. 
Our method shares some similarities with \cite{DG99} which studies the compressible Navier-Stokes equations on the whole space. 
Indeed, there the authors exploit some Strichartz type estimates to analyse the acoustic waves. 
On the other hand, for the QNS system the dispersion relation is modified and reads as in formula \eqref{eq:Bog_sp}; thus for high frequencies the fluctuations appearing in classical fluid dynamics and in system \eqref{eq:QNS1} differ considerably.
Recently, the incompressible limit for a similar system has been investigated in \cite{YY} and later in \cite{KL18}. In both papers the quantum Navier-Stokes system is augmented by adding a damping term in the momentum equation. This extra term allows to circumvent mathematical difficulties related to the lack of control of the velocity field in the vacuum. Moreover both papers deal with smooth local in time solutions for the limiting incompressible dynamics. By using this further regularity assumption it is then possible for them to exploit a relative entropy method.

Here, we tackle the problem from a different perspective, namely we retrieve global weak solutions in the limit rather than convergence to the unique local strong solution to the limiting system. 
Moreover, while in \cite{KL18, YY} the fluctuations are studied by using a wave-like dispersion as for classical fluid dynamical systems, here we consider the full dispersion relation determined by the Bogoliubov spectrum \eqref{eq:Bog_sp} and obtain a better control on the fluctuations. This is achieved by carrying out a refined analysis on the dispersive properties of the acoustic waves that together with new uniform estimates enables us to study the low Mach number limit for general ill-prepared initial data without regularity or smallness assumptions and without damping. 
For the inviscid system, i.e. the QHD system, the low Mach number limit with ill-prepared data has been studied in \cite{DM16} on the torus and on the plane in the forthcoming paper \cite{AHM18}. 

This paper is organized as follows, we introduce notations and preliminary results in Section \ref{sec:pre}. Subsequently, the needed \emph{a priori} estimates are provided in Section \ref{sec:estimates}. This is particularly relevant since finite energy weak solutions of \eqref{eq:QNS} only obey a weak form of the energy inequality, for a detailed discussion see Appendix \ref{app:energy}.  Section \ref{sec:acw} is dedicated to the analysis of the acoustic waves. The strong convergence of finite energy weak solutions of \eqref{eq:QNS} to weak solutions of \eqref{eq:ren_ie} is achieved in Section \ref{sec:conv} by means of an Aubin-Lions compactness argument. Furthermore, we investigate the regularity properties of  the limit ${u}$ and show that $u$ lies in the class of Leray solutions under suitable additional assumptions. Appendix \ref{app:Strichartz} is devoted to the proof of the dispersive estimates.


\section{Preliminaries}\label{sec:pre}

\subsection*{Notations}

We list the notations of function spaces and operators used in the following. We denote
\begin{itemize}
\item the symmetric part of the gradient by $\D u =\frac{1}{2}(\nabla u+(\nabla u)^{T})$ and the asymmetric part by $\A u =\frac{1}{2}(\nabla u-(\nabla u)^{T})$,
\item by $\mathcal{D}(\R_{+}\times \R^3 )$ the space of test functions $C_{c}^{\infty}(\R_{+}\times \R^3 )$ an by $\mathcal{D'}(\R_{+}\times \R^3)$ the space of distributions. The duality bracket between $\mathcal{D}$ and $\mathcal{D'}$ is denoted by $\left\langle \cdot,\cdot\right\rangle$,
\item by $L^{p}(\R^d)$ for $1\leq p \leq \infty$ the Lebesgue space with norm $\|\cdot\|_{L^p}$. We denote by $p'$ the H{\"o}lder conjugate exponent of $p$, i.e. $1=\frac1p+\frac{1}{p'}$, and  for $0<T\leq \infty$ by $L^p(0,T;L^q(\R^d))$ the space of functions $u: (0,T)\times \R^d \rightarrow \R^n$ with norm 
\[
\|u\|_{L^pL^q}=\left(\int_0^{T}\left|\int_{\R^d}|u(t,x)|^q\right|^{\frac{p}{q}}\text{d}x\text{d}t\right)^{\frac{1}{p}}.
\]
If $T=\infty$, we write $L^p(\R_+;L^q(\R^d))$. Further, we denote by $L^{{p-}}(0,T;L^{q}(\R^d))$, the functions $f$ such that $f\in L^{p_0}(0,T;L^{q}(\R^d))$ for any $1\leq p_0<p$, 
\item The sum $L^{q_1}(\R^d)+L^{q_2}(\R^d)$ is a Banach space with norm $\|f\|_{L^{q_1}+L^{q_2}}=\inf\{\|g\|_{L^{q_1}}+\|h\|_{L^{q_2}} \, \, : \, f=g+h, \, g\in L^{q_1}(\R^d), h\in L^{q_2}(\R^d)\}$,
\item by $L_2^p(\R^d)$ the Orlicz space defined as
\begin{equation*}
L_2^p(\R^d)=\left\{f\in L_{loc}^1(\R^d) \,  : \, |f|\chi_{\{|f|\leq \frac{1}{2}\}}\in L^2(\R^d), \, |f|\chi_{\{|f|\geq \frac{1}{2}\}}\in L^p(\R^d)\right\}, 
\end{equation*}
we refer to \cite{AF03,L96} for details.
\item for $s\in \R$ and $p\in[1,\infty]$ the non-homogeneous Sobolev space by $W^{s,p}(\R^d)=(I-\Delta)^{-\frac{s}{2}}L^p(\R^d)$ and $H^s(\R^d)=W^{s,2}(\R^d)$. Its dual will be denoted by $W^{-s,p'}(\R^d)$ with $p'$ being the H\"older conjugate of $p$. The homogeneous spaces are denoted by $\dot{W}^{s,p}(\R^d)=(-\Delta)^{-\frac{k}{s}}L^p(\R^d)$ and $\dot{W}^{s,2}(\R^d)=\dot{H}^{s}(\R^d)$, and the dual space  $\dot{W}^{-s,p'}(\R^d)$. For $ps<d$ we denote the critical Sobolev exponent by $p^{\ast}=\frac{dp}{d-ps}$.
We refer to Theorem 6.4.5 and Theorem 6.5.1 in \cite{BL}, see also Chapter 4 in \cite{AF03}, for the classical embedding results for Sobolev and Lebesgue spaces.
\item by $\Q$ and $\P$ the Helmholtz--Leray projectors on irrotational and divergence-free vector fields, respectively:
\begin{equation*}
\Q=\nabla\Delta^{-1}\dive, \hspace{2cm} \P=\mathbf{I}-\Q.
\end{equation*}
 For $f\in W^{k,p}(\R^d)$ with $1<p<\infty$ and $s\in \R$ the operators $\P,\Q$ can be expressed as composition of Riesz multipliers and are bounded linear operators on $W^{s,p}(\R^d)$.
\item the Fourier transform of $f$ by $\hat{f}:=\mathcal{F}(f)$ and the inverse Fourier transform by $f^{\vee}$,
\item for $R>0$ the frequency cut-off $P_{R}(f)=(\phi_R(\xi)\hat{f})^{\vee}$, where $\phi$ is a smooth frequency cut-off compactly supported in $supp(\phi)\subset\{\frac{1}{2}R\leq |\xi|\leq 2R\}$. Similarly, by $P_{\leq R}(f)$ we denote the projection on frequencies of order $|\xi|\leq R$,
\item 
by $\dot{B}_{q,r}^s(\R^d)$ the homogeneous Besov space. The dual space of  $\dot{B}_{q,r}^{s}(\R^d)$  can be identified with $\dot{B}_{q',r'}^{-s}(\R^d)$, see Chapter 5 of \cite{BL}.
\end{itemize}
In what follows $C$ will be any constant independent from $\eps$. \\
For the convenience of the reader, we recall an interpolation result used several times throughout the paper.

\begin{lemma}[Interpolation]\label{lem:interpolation}
Let $T>0$, $p_1,p_2,q\in(1,\infty)$ and $s_0<s_1$ real numbers. Further, let $u\in L^{p_1}(0,T;W^{s_0,q}(\R^d))\cap L^{p_2}(0,T;W^{s_1,q}(\R^d))$. Then, for all $(p,s)$ such that there exists $\theta\in(0,1)$ with
\begin{equation*}
    \frac{1}{p}=\frac{\theta}{p_1}+\frac{1-\theta}{p_2}, \quad s=\theta s_0+(1-\theta)s_1,
\end{equation*}
it holds $u\in L^p(0,T;W^{s,q}(\R^d))$ with
\begin{equation*}
    \|u\|_{L^p(0,T;W^{s,q}(\R^d))}\leq  \|u\|_{L^{p_1}(0,T;W^{s_0,q}(\R^d))}^{\theta}\|u\|_{L^{p_2}(0,T;W^{s_1,q}(\R^d))}^{1-\theta}.
\end{equation*}
\end{lemma}
The Lemma is a simplified statement of Theorem 5.1.2 in \cite{BL} applied with $A_0=W^{s_0,q}(\R^d)$ and $A_1=W^{s_1,q}(\R^d)$. The Lemma can also be proven by standard interpolation of Sobolev spaces in the space variables, see e.g. Paragraph 7.53 in \cite{AF03},  followed by H{\"o}lder's inequality in the time variable.
\subsection{Scaling}\label{subsec:scaling}
To recast the introduced scaling \eqref{eq:QNS} of system \eqref{eq:QNS1}, one starts writing the equations by re-scaling each length scale by its characteristic value (dimensionless scaling) and we assume the Mach number to be small. We expect the fluid to behave like an incompressible fluid on large time scales when the density is almost constant and the the velocity is small. Thus, we introduce the change of variable and unknowns,
\begin{equation*}
    t \mapsto \eps t_{\eps}, \qquad u\mapsto \eps u_{\eps}.
 \end{equation*}
Moreover, the viscosity and capillarity coefficients scale as 
\begin{equation*}
\nu \mapsto \eps \nu_{\eps}, \qquad \kappa \mapsto \eps \kappa_{\eps},
\end{equation*}
where
\[
\nu_{\eps} \rightarrow \tilde\nu>0, \qquad \kappa_{\eps}\rightarrow \tilde\kappa>0,
\]
as $\eps$ goes to $0$. We refer the reader to the review papers \cite{A,M} for a more detailed discussion of suitable low Mach number scalings.
\subsection*{Weak solutions}

As we already mentioned in the Introduction, the degenerate viscosity prevents the velocity field to be uniquely determined in the vacuum region; indeed system \eqref{eq:QNS} lacks bounds for $\ueps$. Consequently,
in this framework (see for example \cite{AS, LV16}) it turns out that the problem is best studied in terms of the more suitable variables $\sqrt{\rho_\eps}$ and $\Lambda_\eps=\sqrt{\rho_\eps}u_\eps$. 
In fact, this occurs also when studying the QHD system \cite{AM1} and the barotropic compressible Navier-Stokes equations with degenerate viscosities \cite{LX}.
Mathematically speaking, this means that whenever the symbol $\rho_\eps$ appears, it should be read as $\rho_\eps=(\sqrt{\rho_\eps})^2$ and similarly for the momentum density $m_\eps=\reps\ueps=\sqrt{\rho_\eps}\Lambda_\eps$. At no moment, the velocity field $\ueps$ or its gradient $\nabla\ueps$ are defined a.e. in $\R^3$. For those reasons, the viscous tensor should be rather thought as
\begin{equation}\label{eq:visc_symm}
\rho_\eps\D u_\eps=\sqrt{\rho_\eps}\mathbf S_\eps,
\end{equation}
where $\mathbf S_\eps$ is the symmetric part of the tensor $\mathbf T_\eps$ defined through the following identity
\begin{equation}\label{eq:visc_tens}
\sqrt{\rho_\eps}\mathbf T_\eps=\nabla(\rho_\eps u_\eps)-2\Lambda_\eps\otimes \nabla\sqrt{\rho_\eps},
\end{equation}
in $\mathcal{D}'((0,T)\times\R^3)$. 
In our paper, we do not use the notation $\Lambda_{\eps}$ instead of $\sqrt{\reps}\ueps$ for the sake of consistency with the literature regarding (quantum) Navier-Stokes equations. The definition of finite energy weak solutions will therefore be given in terms of the mathematical unknowns $\sqrt{\reps}$ and $\sqrt{\reps}\ueps$ instead of the physical unknowns of density $\reps$ and momentum $\meps$. We recall that under suitable assumptions on the mass density $\rho$ the quantum pressure term can be alternatively rewritten as
\begin{equation}\label{eq:Bohm}
2\rho \nabla\left(\frac{\Delta\sqrt{\rho}}{\sqrt{\rho}}\right)=\dive\left(\rho\nabla^2\log\rho\right)=\nabla\Delta\rho-4\dive(\nabla\sqrt{\rho}\otimes \nabla\sqrt{\rho}).
\end{equation}
For the sake of simplifying our exposition, from now on we will suppress the $\eps$-dependence of $\nu$ and $\kappa$.
In this way the equation for the momentum density in \eqref{eq:QNS} reads
\begin{equation*}
\begin{aligned}
\partial_t(\rho_{\eps}u_{\eps})+\dive\left(\rho_{\eps}u_{\eps}\otimes u_{\eps}+4\kappa^2\nabla\sqrt{\rho_{\eps}} \otimes \nabla\sqrt{\rho_{\eps}}\right)&+\frac{1}{\eps^2}\nabla P(\rho_{\eps})\\
&=2\nu \dive(\sqrt{\rho_{\eps}}\mathbf S_\eps)+\kappa^2\nabla\Delta\rho_{\eps}.
\end{aligned}
\end{equation*}

\begin{definition}[Finite energy weak solutions]\label{defi:FEWS}
A pair $(\reps,\ueps)$ with $\reps\geq 0$ is said to be a finite energy weak solution of the Cauchy Problem \eqref{eq:QNS} if
\begin{enumerate}[(i)]
\item integrability conditions
\begin{equation*}
\begin{aligned}
&\sqrt{\reps}\in L_{loc}^{2}((0,T)\times \R^3);
&\sqrt{\reps}\ueps\in L_{loc}^{2}((0,T)\times \R^3);\\
&\nabla\sqrt{\rho_\eps}\in L_{loc}^{2}((0,T)\times \R^3);
\end{aligned}
\end{equation*}
\item continuity equation
\begin{equation*}
\int_{\R^3}\rho_{\eps}^0\phi(0)+\int_0^T\int_{\R^3}\reps \phi_t+\sqrt{\reps}\sqrt{\reps}\ueps\nabla\phi=0,
\end{equation*}
for any $\phi\in C_c^{\infty}([0,T)\times \R^3)$.
\item momentum equation
\begin{equation*}
\begin{aligned}
&\int_{\R^3}\rho_{\eps}^0u_{\eps}^0\psi(0)+\int_0^T\int_{\R^d}\sqrt{\reps}\sqrt{\reps}\ueps \psi_t+(\sqrt{\reps}\ueps\otimes\sqrt{\reps}\ueps)\nabla\psi+\frac{1}{\eps^2}\reps^{\gamma}\dive \psi\\
&-2\nu \int_0^T\int_{\R^3}\left(\sqrt{\reps}\ueps\otimes \nabla \sqrt{\reps}\right)\nabla \psi-2\nu \int_0^T\int_{\R^3}\left(\nabla\sqrt{\reps}\otimes \sqrt{\reps}\ueps\right)\nabla \psi\\
&+\nu \int_0^T\int_{\R^3}\sqrt{\reps}\sqrt{\reps}\ueps\Delta\psi+\nu \int_0^T\int_{\R^3}\sqrt{\reps}\sqrt{\reps}\ueps\nabla\dive\psi\\
&-4\kappa^2 \int_0^T\int_{\R^3}\left(\nabla\sqrt{\reps}\otimes \nabla\sqrt{\reps}\right)\nabla\psi+2\kappa^2\int_0^T\int_{\R^3}\sqrt{\reps}\nabla\sqrt{\reps}\nabla\dive\psi=0,
\end{aligned}
\end{equation*}
for any $\psi\in C_c^{\infty}([0,T)\times \R^3;\R^3)$.
\item there exists a tensor $\mathbf T_\eps\in L^2((0,T)\times\R^3)$  satisfying identity \eqref{eq:visc_tens} in $\mathcal D'((0, T)\times\R^3)$ such that for
\begin{equation*}
    E(t)=\int_{\R^3}\frac{1}{2}\left|\sqrt{\reps}\ueps\right|^2+2\kappa^2\left|\nabla\sqrt{\reps}\right|^2+\pi_{\eps}(\reps)\dd x,
\end{equation*}
the following energy inequality holds for a.e. $t\in[0,T]$,
\begin{equation}\label{ineq:energy}
E(t)+2\nu \int_0^t \int_{\R^3}|\mathbf S_\eps|^2\dd x\dd t\leq E(0),
\end{equation}
where $\mathbf S_\eps$ is the symmetric part of $\mathbf T_\eps$, i.e. $\mathbf S_\eps=\mathbf T_\eps^{sym}$.
\end{enumerate}
\end{definition}
The definition of finite energy weak solutions including \eqref{ineq:energy} is motivated by the degeneracy of the viscosity. While for smooth solutions of \eqref{eq:QNS} with $\reps>0$ the tensor $\mathbf S_\eps$ is equivalent to $\sqrt{\reps}\D\ueps$, this information may not be recovered for finite energy weak solutions. In fact, at present it is not clear  whether arbitrary finite energy weak solutions to \eqref{eq:QNS} satisfy the following energy inequality
\begin{equation}\label{eq:en_diss}
E(t)+2\nu\int_0^t\int_{\R^3}\reps|\D\ueps|^2\dd x \dd t' \leq E(0).
\end{equation}
Therefore, we only assume the weaker version of the energy inequality \eqref{ineq:energy}. The aforementioned Bresch-Desjardins entropy further leads us to define the following notion of BD-entropic solutions. 

\begin{definition}\label{defi:BD_entr}
Let $(\rho_\eps, u_\eps)$ be a finite energy weak solution to \eqref{eq:QNS} as in Definition \ref{defi:FEWS}, we say that $(\rho_\eps, u_\eps)$ is a BD-entropic weak solution if there exists $c>0$ such that defining the BD entropy as
\begin{equation*}
B_{\eps}(t)=\int_{\R^3}\frac{1}{2}\left|\sqrt{\reps}\ueps+2c\nabla\sqrt{\reps}\right|^2+\left|\nabla\sqrt{\reps}\right|^2+\pi_{\eps}(\reps)\dd x,
\end{equation*}
then the Bresch-Desjardins entropy inequality holds for a.e. $t\in[0,T]$,
\begin{align}\label{ineq:BDE}
\begin{split}
&B_{\eps}(t)
+\nu\int_0^t\int_{\R^3}\left|\mathbf A_{\eps}\right|^2\dd x\dd s+\nu\int_0^t\int_{\R^3}\left|\nabla^2\sqrt{\rho_{\eps}}\right|^2\dd x\dd s+\frac{\nu}{4\gamma\eps^2}\int_0^t\int_{\R^3}\left|\nabla\reps^{\frac{\gamma}{2}}\right|^2\dd x\dd s\\
&\leq C B_{\eps}(0),
\end{split}
\end{align}
where $\mathbf A_\eps=\mathbf T_\eps^{asym}$, with $\mathbf T_\eps$ satisfying identity \eqref{eq:visc_tens}.
\end{definition}

Let us further comment on Definition \ref{defi:FEWS} and \ref{defi:BD_entr} respectively. The existence of global in time BD-entropic weak solutions to equation \eqref{eq:QNS1}, posed in $\R^d$, $d=2,3$ with non-trivial far-field behavior for the mass density, has recently been proved in \cite{AHS20}, see also \cite{AS, LV16, AS19} for similar results on the periodic domain $\T^d$. More specifically, in \cite{AHS20} it has been proven that, for any finite energy initial data, there exists a global solution satisfying the Definitions \ref{defi:FEWS} and \ref{defi:BD_entr} above. 
Their proof exploits the construction of a sequence of approximating solutions by solving the Cauchy problem on a sequence of invading domains.
One of the key points in \cite{AHS20} is that the approximating solutions satisfy both the energy and BD entropy estimates, which allow to obtain the necessary compactness and to pass to the limit. Therefore, the solutions studied satisfy the bounds \eqref{ineq:energy} and \eqref{ineq:BDE} by construction. 
Let us remark that in general it is not true that weak solutions to \eqref{eq:QNS1} satisfy the energy inequality \eqref{ineq:energy} and the BD entropy estimate \eqref{ineq:BDE}. Moreover, it is an open problem to determine the minimal assumptions on weak solutions to \eqref{eq:QNS1} in order to satisfy the two aforementioned inequalities.
Presently, a physical interpretation of the BD type entropies is not yet clear, however this kind of estimates are naturally satisfied by a large class of weak solutions to compressible, viscous, fluid dynamic equations with a degenerate viscosity coefficient. 
At the formal level, they can be interpreted as the energy estimates associated to an equivalent system to \eqref{eq:QNS1}, reformulated in terms of an effective velocity involving the density gradient \cite{JqNS}, see also system \eqref{eq:BD}.  For the convenience of the reader, in Appendix \ref{app:energy} we provide a class of weak solutions to \eqref{eq:QNS1} which satisfies \eqref{ineq:energy} and \eqref{ineq:BDE}. More precisely, we give a direct proof that the weak solutions on $\T^d$ constructed in \cite{AS} as limit of smooth approximating solutions satisfy the aforementioned inequalities. Let us also mention that it is possible to adopt different approximation procedures to show the existence of weak solutions satisfying \eqref{ineq:energy} and \eqref{ineq:BDE}, for example a suitable strategy would be the one adopted in \cite{AHS20}, see also \cite{AS18} where a similar issue is dealt with for a Navier-Stokes-Korteweg type system and \cite{CKH} where a further approximation method is used in the isothermal case.
We stress that despite the fact that only \eqref{ineq:energy} and \eqref{ineq:BDE} are available for the compressible system \eqref{eq:QNS}, we achieve convergence to a Leray weak solution for well-prepared initial data, see Proposition \ref{prop:leray}.

\subsection*{Main result}
Let us specify the assumptions on the initial data for the system \eqref{eq:QNS}. We consider initial data $(\rho_{\eps}^0,u_{\eps}^0)$ of finite energy, namely such that 
\begin{equation}\label{eq:ID1}
    \|\nabla\sqrt{\reps^0}\|_{L^2(\R^3)}\leq C, \quad \|\sqrt{\reps^0}\ueps^0\|_{L^2(\R^3)}\leq C, \quad \|\pi_{\eps}(\reps^0)\|_{L^1(\R^3)}\leq C,
\end{equation}
where is $C$ independent on $\eps>0$. Furthermore, we assume that 
\begin{equation}\label{eq: ID}
\sqrt{\reps^{0}}\ueps^{0}\rightharpoonup u_0 \quad \text{in} \, \, L^2(\R^3).
\end{equation}

 With this definition at hand, we now state the main Theorem characterising the low Mach number regime for \eqref{eq:QNS}.
\begin{theorem}\label{thm:main}
{Let $\gamma>1$} and let $(\reps,\ueps)$  be a BD-entropic weak solution of \eqref{eq:QNS} with initial data satisfying \eqref{eq:ID1} and \eqref{eq: ID} and let $0<T<\infty$ be an arbitrary time. Then $\reps-1$ converges strongly to $0$ in $L^{\infty}(0,T;L^{2}(\R^3))\cap L^4(0,T;H^s(\R^3))$ for any $0\leq s<1$. For any subsequence (not relabeled) $\sqrt{\reps}\ueps$ converging weakly to ${u}$ in $L^{\infty}(0,T,L^2(\R^3))$, then $u\in L^{\infty}(0,T;L^2(\R^3))\cap L^2(0,T;\dot{H}^1(\R^3))$ is a global weak solution to the incompressible Navier-Stokes equation \eqref{eq:INS} with initial data ${u}_{\big| t=0}=\mathbf{P}(u_0)$ and $\sqrt{\reps}\ueps$ converges strongly to $u$ in $L^2(0,T;L_{loc}^2(\R^3))$.\\
Moreover, $\Q(\reps\ueps)$ converges strongly to $0$ in $L^2(0,T,L^q(\R^3))$ for any $2<q<\frac{12}{5}$.
Finally the limiting solution $u$ also satisfies $u\in L^{\frac{4}{1+4s}-}(0,T;H^s(\R^3))$, for $0\leq s\leq\frac12$.
\end{theorem}

The existence of global in time BD-entropic results to \eqref{eq:QNS} posed on $\R^d$, $d=2,3$ with far-field \eqref{eq:farfield} has recently been introduced in \cite{AHS20}, see also \cite{AS, LV16} for related results on $\T^d$ with $d=2,3$.

\begin{remark}
Let us remark that in order for the limiting function ${u}$ to satisfy the energy inequality, i.e. to be a Leray weak solution \cite{L34}, stronger assumptions on the initial data $(\reps^0,\ueps^0)$ are needed. 
Indeed the initial total energy for the compressible system in general does not converge, as $\eps\to0$, to the initial energy for \eqref{eq:INS}, which would be given by 
$\frac12\int|\mathbf P u_0|^2$. The excess energy determines an initial layer which cannot be avoided for ill-prepared data. On the other hand, if we require
\begin{equation}\label{eq:wellprepared}
\begin{aligned}
&\sqrt{\reps^0}\ueps^0\rightarrow u_0=\P(u_0) \quad \text{strongly in} \quad L^2(\R^3),\\
&\pi_{\eps}(\reps^0)\rightarrow 0 \quad \text{strongly in} \quad L^1(\R^3),\\
&\nabla\sqrt{\reps^0}\rightarrow 0  \quad \text{strongly in} \quad L^2(\R^3),
\end{aligned}
\end{equation}
then the following Proposition holds true.
\end{remark}

\begin{proposition}\label{prop:leray}
Under the same assumptions of Theorem \ref{thm:main}, let $(\reps^0,\ueps^0)$ further satisfy \eqref{eq:wellprepared}.
Then the limiting solution $u$ to \eqref{eq:INS} satisfies the energy inequality
\begin{equation}\label{ineq:Leray}
    \int_{\R^3}\frac{1}{2}|u(t)|^2\dd x+\nu \int_0^t\int_{\R^3}|\nabla u|^2\dd x \dd t'\leq \int_{\R^3}\frac{1}{2} |u_0|^2\dd x,
\end{equation}
for almost every $t\in [0,T]$.
\end{proposition}


\section{Uniform estimates}\label{sec:estimates}

In this Section, we start our analysis on the low Mach number limit by inferring some uniform estimates for finite energy weak solutions to \eqref{eq:QNS}. In our framework we need to take the non trivial boundary conditions for the mass density into account. For this reason, we provide some estimates on the quantities $\sqrt{\reps}-1$ and $\reps-1$. Furthermore, the lack of control for $\nabla\ueps$ in the vacuum will be compensated by the bounds inferred from \eqref{ineq:energy} and \eqref{ineq:BDE}.

\begin{remark}\label{rem:integrability}
We shall use repeatedly the following observation, see for instance Theorem 4.5.9 in \cite{H90}: if $f\in \mathcal{D'}(\R^d)$ with $\nabla f\in L^{p}(\R^d)$ for $p<d$, then there exists a constant $c$ such that $f-c\in L^{p^{\ast}}(\R^d)$, where $p^{\ast}=\frac{pd}{d-p}$. The condition $p<d$ is sharp. 
\end{remark}
If $(\reps,\ueps)$ is of finite energy, i.e. $E(\reps,\ueps)<\infty$, then
\begin{equation*}
\nabla\sqrt{\reps}\in L^2(\R^3), \hspace{0.9cm}\sqrt{\rho_{\eps}}u_{\eps}\in L^2(\R^3), \hspace{0.9cm} \pi_{\eps}(\reps)=\frac{\reps^{
\gamma}-1-\gamma(\reps-1)}{\eps^2\gamma(\gamma-1)}\in L^1(\R^3).
\end{equation*}
This implies additional bounds which we summarize.
\begin{lemma}\label{lem:initialdata}
Let $C>0$ independent from $\eps>0$ and $(\reps,\ueps)$ be such \linebreak[4] that $E(\reps,\ueps)\leq C$,
then
\begin{enumerate}[(i)]
    \item $\reps$ converges strongly to $1$ in $L^2(\R^3)$, more precisely
    \begin{equation}\label{eq:alpha}
        \|\reps-1\|_{L^2(\R^3)}\leq C\eps^{\alpha}, \qquad \text{where} \quad \alpha(\gamma)=\begin{cases}\frac{2}{6-\gamma}, \quad &1<\gamma<2, \\ 1 \quad  &\gamma\geq 2. \end{cases}
    \end{equation}
    \item $\sqrt{\reps}-1\in H^1(\R^3)$ is uniformly bounded. Moreover,
    \begin{equation*}
        \|\sqrt{\reps}-1\|_{L^2(\R^3)}\leq C\eps^{\beta}, \qquad \text{where} \quad \beta(\gamma)=\begin{cases}\frac{4}{6-\gamma}, \quad &1<\gamma<2, \\ 1 \quad &\gamma\geq 2. \end{cases}
    \end{equation*}
    \item $\reps\ueps\in L^{\frac{3}{2}}(\R^3)+L^2(\R^3)$ uniformly bounded. In particular $\reps\ueps\in H^{-s}(\R^3)$ uniformly bounded for all $s\geq\frac{1}{2}$.
\end{enumerate}
\end{lemma}

We remark that the decay rate of (i) is not optimal but sufficient for our purpose.

\begin{proof}
We recall the argument in \cite{LM98} relating the integrability of the internal energy to the Orlicz spaces $L_2^{\gamma}(\R^3)$. Exploiting the convexity of 
\begin{equation*}
    s\mapsto s^{\gamma}-1-\gamma(s-1),
\end{equation*}
for $\gamma>1$ one obtains from the uniform bound of the energy functional $E$ that 
\begin{equation}\label{eq:strongreps0}
\int_{\R^3}\left|\reps-1\right|^{2} \mathbf{1}_{|\reps-1|\leq\frac{1}{2}}+\left|\reps-1\right|^{\gamma} \mathbf{1}_{|\reps-1|>\frac{1}{2}}\text{d}x\leq C\eps^2 \int_{\R^3}\pi_{\eps}(\reps)\dd x\leq C\eps^2,
\end{equation}
where $\pi_{\eps}(\reps)$ is defined in \eqref{eq:ren_ie}. 
It follows in particular that $\reps-1\in L_2^{\gamma}(\R^3)$ uniformly bounded for $\gamma>1$ and 
\begin{equation*}
    \|\reps-1\|_{L^2(\R^3)}\leq C\eps,
\end{equation*}
provided $\gamma\geq 2$. Next, we show (ii). It follows from Remark \ref{rem:integrability} and $\nabla\sqrt{\reps}\in L^2(\R^3)$ that there exists $c\in \R$ such that $\sqrt{\reps}-c\in L^{6}(\R^3)$ uniformly  bounded. We notice that 
\begin{equation*}
    \left|\sqrt{\reps}-1\right|\leq \left|\frac{{\reps}-1}{\sqrt{\reps}+1}\right|\leq \left|\reps-1\right|,
\end{equation*}
and therefore \eqref{eq:strongreps0} implies that 
\begin{equation}\label{eq:strongreps01}
\int_{\R^3}\left|\sqrt{\reps}-1\right|^{2} \mathbf{1}_{|\reps-1|\leq\frac{1}{2}}+\left|\sqrt{\reps}-1\right|^{\gamma} \mathbf{1}_{|\reps-1|>\frac{1}{2}}\text{d}x\leq C\eps^2.
\end{equation}
It follows that $\sqrt{\reps}-1\in L_2^{\gamma}(\R^3)$ uniformly bounded. In particular, for $\gamma\geq 2$ we have
\begin{equation*}
    \|\sqrt{\reps}-1\|_{L^2(\R^3)}\leq C\eps,
\end{equation*}
and $\sqrt{\reps}-1\in H^1(\R^3)$ provided $\gamma\geq 2$. For $1<\gamma<2$, we infer the desired decay by interpolation. First, we notice that $\sqrt{\reps}-1\in L^6(\R^3)$ uniformly, namely $c=1$. This easily follows from \eqref{eq:strongreps01} and Markov inequality. Therefore, for $1<\gamma<2$ and $\theta=\frac{2\gamma}{6-\gamma}$ we obtain by interpolation that
\begin{equation*}
    \|\sqrt{\reps}-1\|_{L^2(\{|\reps-1|\geq \frac{1}{2}\})}\leq C \|\sqrt{\reps}-1\|_{L^{\gamma}(\{|\reps-1|\geq \frac{1}{2}\})}^{\theta}\|\sqrt{\reps}-1\|_{L^6(\{|\reps-1|\geq \frac{1}{2}\})}^{1-\theta}\leq C \eps^{\frac{2}{\gamma}\theta}.
\end{equation*}
This estimate together with \eqref{eq:strongreps01} yields 
\begin{equation}\label{eq:decaysqrtreps}
    \|\sqrt{\reps}-1\|_{L^2(\R^3)}\leq C\eps^{\beta}, \quad \beta=\begin{cases} \frac{4}{6-\gamma}, \quad &1<\gamma<2,\\ 1 \quad &\gamma\geq 2.\end{cases}
\end{equation}
We proceed to show (i) for $1<\gamma<2$. We observe that 
\begin{equation*}
    (z^2-1)^2\leq 25(z-1)^4,
\end{equation*}
for $|z-1|\geq \frac{1}{2}$. Hence by interpolation between \eqref{eq:decaysqrtreps} and the uniform $L^6$-bound for $\sqrt{\reps}-1$ of the previous step, we obtain
\begin{equation*}
    \int_{\R^3}\left|\reps-1\right|^{2} \mathbf{1}_{|\reps-1|>\frac{1}{2}}\text{d}x\leq C \int_{\R^3}\left|\sqrt{\reps}-1\right|^{4} \mathbf{1}_{|\reps-1|>\frac{1}{2}}\dd x\leq C\eps^{\frac{4}{6-\gamma}}.
\end{equation*}
It follows,
\begin{equation*}
    \|\reps-1\|_{L^2(\R^3)}\leq C\eps^{\alpha}, \quad \text{for} \quad \alpha=\begin{cases} \frac{2}{6-\gamma}, \quad &1<\gamma<2,\\ 1 \quad &\gamma\geq 2.\end{cases}
\end{equation*}
Finally, to prove (iii) we notice that if $(\reps,\ueps)$ is of finite energy then the bounds $\sqrt{\reps}-1\in L^6(\R^3)$ from (ii) and $\sqrt{\reps}\ueps\in L^2(\R^3)$ allow us to conclude 
\begin{equation*}
\reps\ueps=\sqrt{\reps}\ueps+(\sqrt{\reps}-1)\sqrt{\reps}\ueps\in L^2(\R^3)+L^{\frac{3}{2}}(\R^3).
\end{equation*}
As $L^{\frac{3}{2}}(\R^3)\hookrightarrow H^{-s}(\R^3)$ continuously for $s\geq \frac{1}{2}$, statement (iii) is proven. This completes the proof.
\end{proof}

\subsection{Uniform estimates on the solution}\label{sub:ee}

By Definition \ref{defi:BD_entr}, the BD-entropic weak solution $(\reps,\ueps)$ we consider satisfies the energy inequality \eqref{ineq:energy} and the BD entropy type inequality  \eqref{ineq:BDE} that imply the following \emph{a priori}  estimates listed below.

\begin{lemma}\label{lem:apriori}
If $(\reps,\ueps)$ is a BD-entropic weak solution of \eqref{eq:QNS}, then for any $0<T<\infty$ there exists $C>0$ independent from $\eps>0$ such that
\begin{enumerate}[(i)]
\item $\|\reps-1\|_{L^{\infty}(0,T; L^2(\R^3))}\leq C\eps^{\alpha}$,
for $\alpha(\gamma)$  defined as in \eqref{eq:alpha}
\item $\|\frac{1}{\eps}\nabla \reps^{\frac{\gamma}{2}}\|_{L^{2}(0,T;L^2(\R^3))}\leq C$ ,
\item $\sqrt{\reps}-1\in L^{\infty}(0,T; H^1(\R^3))$ uniformly and in particular for $2\leq p< 6$ 
\begin{equation*}
    \|\sqrt{\reps}-1\|_{L^{\infty}(0,T;L^p(\R^3))}\leq C\eps^{\beta(p)}, \qquad \beta(p)=\frac{2(6-p)}{p(6-\gamma)},
\end{equation*}
\item $\|\nabla^2\sqrt{\rho_{\eps}}\|_{ L^{2}({0,T};L^2(\R^3))}\leq C$, 
\item  for any $0< s<2$ and $2\leq p<\frac{4}{s}$, there exists $0<\beta(p,s)< 2$ such that  \mbox{$\|\sqrt{\reps}-1\|_{L^p(0,T;H^s(\R^3))}\leq C\eps^{\beta}$}.
Moreover, for $1< s \leq 2$,\\
 \mbox{$\|\sqrt{\reps}-1\|_{L^{\frac{2}{s-1}}(0,T;H^s(\R^3))}\leq C $.} In particular, $\sqrt{\reps}-1\in L^{4-}({0,T};L^{\infty}(\R^3))$ uniformly bounded.
\item $\|\sqrt{\rho_{\eps}}u_{\eps}\|_{ L^{\infty}(0,T;L^2(\R^3))}\leq C$,
\end{enumerate}
\end{lemma}

\begin{proof}
From \eqref{ineq:energy}, one has that $\sup_{t\geq 0} E(\reps(t),\ueps(t))\leq C$, hence the bounds of Lemma \ref{lem:initialdata} are uniform in $t\geq 0$. We notice that (i) and the first statement of (iii) are immediate consequences of (i) and (ii) of Lemma \ref{lem:initialdata}. The second part of (iii) follows by Sobolev embedding and interpolation with (ii) of Lemma \ref{lem:initialdata} being valid uniformly in time. More precisely, for $p\in(2,6)$ and $\theta=3\left(\frac{1}{p}-\frac{1}{6}\right)$ one has 
\begin{equation*}
    \|\sqrt{\reps}-1\|_{L^{\infty}(0,T;L^p(\R^3))}\leq C \|\sqrt{\reps}-1\|_{L^{\infty}(0,T;L^2(\R^3))}^{\theta}\|\sqrt{\reps}-1\|_{L^{\infty}(0,T;L^6(\R^3))}^{1-\theta}\leq C \eps^{2\theta\alpha},
\end{equation*}
with $\alpha$ as defined in (ii) of Lemma \ref{lem:initialdata}. The remaining statements except from (v) are direct consequences of inequalities \eqref{ineq:energy} and \eqref{ineq:BDE}. Statement (v) follows by interpolation of (iii) and (iv). We notice that (iii) and (iv) yield $\sqrt{\reps}-1\in L^2(0,T;H^2(\R^3))$ uniformly. We apply Lemma \ref{lem:interpolation} to conclude (v). For $0<s<2$ and $0<\theta<1$ there exists $\beta>0$ such that
\begin{equation*}
\left\|\sqrt{\reps}-1\right\|_{L^{p}({0,T};H^s(\R^3))}\leq C \left\|\sqrt{\reps}-1\right\|_{L^{2}({0,T};H^2(\R^3))}^{\theta} \left\|\sqrt{\reps}-1\right\|_{L^{\infty}(0,T;L^2(\R^3))}^{1-\theta}\leq C \eps^{\beta},
\end{equation*}
where $s=2\theta$ and $p$ such that $p\leq \frac{2}{\theta}$. In particular if $s>\frac{3}{2}$ this yields for any $2\leq p <\frac{8}{3}$ that $\sqrt{\reps}-1$ converges strongly to $0$ in $L^{p}(0,T;L^{\infty}(\R^3))$. By interpolation between the bounds $\sqrt{\reps}-1\in L^{\infty}(0,T;H^1(\R^3))$ and $\nabla^2(\sqrt{\reps}-1)\in L^{2}(0,T; L^2(\R^3))$, one may infer the slightly stronger bound $\sqrt{\rho_{\eps}}-1\in L^{\frac{2}{s-1}}(0,T; H^s(\R^3))$ for  $1< s\leq 2$.
\end{proof}

\subsection{Bounds on density fluctuations and momenta}
First, we provide bounds on the density fluctuation $\seps:=\frac{\reps-1}{\eps}$.

\begin{lemma}\label{lem:fluctuations}
If $(\reps,\ueps)$ is a BD-entropic weak solution of \eqref{eq:QNS}, then for any $0<T<\infty$, the fluctuations $\seps$ satisfy the following
\begin{enumerate}[(i)] 
\item $\seps^0\in L_{2}^\gamma(\R^3)$ uniformly bounded and in particular $\seps^0\in L^2(\R^3)$ for $\gamma\geq 2$ and $\seps^0\in H^{-\frac{3}{2}}(\R^3)$ for $\gamma\in(1,2)$ uniformly bounded,
\item  $\eps\nabla\seps^0\in H^{-\frac{1}{2}}(\R^3)$ uniformly bounded,
\item $\seps\in L^{\infty}(0,T;L_2^\gamma(\R^3))$ uniformly bounded, in particular $\seps\in  L^{\infty}(0,T;L^{2}(\R^3))$ for $\gamma\geq 2$ and $\seps\in  L^{\infty}(0,T;H^{-\frac{3}{2}}(\R^3))$ for $\gamma\in(1,2)$ uniformly bounded,
\item $\eps\seps\in L^{4-}(0,T;H^1(\R^3))$ uniformly bounded,
\item if $\gamma=2$, then $\seps\in L^2(0,T,H^1(\R^3))$ uniformly bounded.
\end{enumerate}
\end{lemma}

\begin{proof}
For (i), we notice that $\seps^0=\frac{\reps^0-1}{\eps}$. It follows from Lemma \ref{lem:initialdata} that $\seps^0\in L_2^\gamma(\R^3)$ uniformly bounded. If $\gamma\geq 2$, then $\seps^0\in L^2(\R^3)$ uniformly bounded. If $\gamma\in (1,2)$, we observe that $\seps^0\in L_2^{\gamma}(\R^3)$ uniformly bounded implies in particular $\seps^0\in L^2(\R^3)+L^{\gamma}(\R^3)$ uniformly. Since $L^{\gamma}(\R^3)\hookrightarrow H^{-\frac{3}{2}}(\R^3)$ for $\gamma\in(1,2]$, we obtain $\seps^0\in H^{-\frac{3}{2}}(\R^3)$. To show (ii), we observe that
\begin{equation*}
    \eps\nabla\seps^0=2\sqrt{\reps^0}\nabla\sqrt{\reps^0}=2\nabla\sqrt{\reps^0}+2(\sqrt{\reps^0}-1)\nabla\sqrt{\reps^0}\in L^2(\R^3)+L^{\frac{3}{2}}(\R^3),
\end{equation*}
in virtue of (iii) of Lemma \ref{lem:apriori}. Hence $\eps\nabla\seps^0\in H^{-\frac{1}{2}}(\R^3)$. The first part of (iii) follows from \eqref{eq:strongreps0}. The second part follows from the Sobolev embedding upon observing that $L_2^\gamma(\R^3)\subset L^2(\R^3)$ for $\gamma\geq 2$ and $L_2^\gamma(\R^3)\subset L^2(\R^3)+L^{\gamma}(\R^3)\hookrightarrow H^{-\frac32}(\R^3)$ otherwise. To show (iv), we observe that $\eps\seps\in L^{\infty}(0,T;L^2(\R^3))$ from (i) of Lemma \ref{lem:apriori}. By exploiting (iii) and (v) of Lemma \ref{lem:apriori}, we obtain
\begin{equation*}
\begin{aligned}
&\|\eps\nabla\seps\|_{L^{4-}(0,T;L^2(\R^3))}\\
&\leq C_T \|\nabla\sqrt{\reps}\|_{L^{\infty}(0,T;L^2(\R^3))}+2\|\sqrt{\reps}-1\|_{L^{4-}(0,T;L^{\infty}(\R^3))}\|\nabla\sqrt{\reps}\|_{L^{\infty}(0,T;L^2(\R^3))}.
\end{aligned}
\end{equation*}
Moreover, for $\gamma=2$, the estimate $\frac{1}{\eps}\nabla\reps^{\frac{\gamma}{2}}\in L^2(0,T;L^2(\R^3))$ provided by \eqref{ineq:BDE} allows us to conclude that  $\seps\in L^2(0,T;H^1(\R^3))$.
\end{proof}

\begin{corollary}\label{coro:integrabilitym}
If $(\reps,\ueps)$ is a BD-entropic weak solution of \eqref{eq:QNS}, then for any $0<T<\infty$,
\begin{equation*}
\reps u_{\eps}\in L^{4-}(0,T;L^2(\R^3)).
\end{equation*}
\end{corollary}

\begin{proof}
It is sufficient to write $\reps u_{\eps}=\sqrt{\reps}u_{\eps}+(\sqrt{\reps}-1)\sqrt{\reps}u_{\eps}$ and to see that
\begin{equation*}
\begin{aligned}
\|\reps u_{\eps}\|_{L^{4-}(0,T;L^2(\R^3))}&\leq C \| \sqrt{\reps}u_{\eps}\|_{L^{\infty}(0,T;L^2(\R^3))}\left(T^{\frac{1}{4}+}+\|\sqrt{\reps}-1\|_{L^{4-}(0,T;L^{\infty}(\R^3))}\right)\\
&\leq C' (1+T^{\frac{1}{4}+}),
\end{aligned}
\end{equation*}
since $\sqrt{\reps}-1\in L^{4-}(0,T;L^{\infty}(\R^3))$ uniformly in view of (v) from Lemma \ref{lem:apriori}.
\end{proof}

The Corollary \ref{coro:integrabilitym} together with the a priori bounds on $\sqrt{\reps}-1$ and $\Teps$ allow us to prove a stronger estimate on $\reps u_{\eps}$.

\begin{proposition}\label{prop:boundm}
If $(\reps,\ueps)$ is a BD-entropic weak solution of \eqref{eq:QNS}, then for any $0<T<\infty$, $0\leq s\leq \frac{1}{2}$ and $1\leq p<\frac{4}{1+4s}$,
\begin{equation}\label{eq:boundm}
\reps u_{\eps}\in L^p(0,T;H^{s}(\R^3)),
\end{equation}
uniformly in $\eps>0$.\\
In particular, for any $0\leq s_1<\frac{1}{4}$ one has $\reps u_{\eps}\in L^2(0,T;H^{s_1}(\R^3))$.
\end{proposition}

\begin{proof}
From \eqref{eq:visc_tens}, one has
\begin{equation*}
    \nabla(\reps\ueps)=F_{\eps},
\end{equation*}
in distributional sense where $\Feps:=\Feps^1+\Feps^2$ with
\begin{equation*}
    \Feps^1=2\sqrt{\reps}\ueps\otimes \nabla\sqrt{\reps}, \qquad \Feps^2=\Teps+(\sqrt{\reps}-1)\Teps.
\end{equation*}
We observe that $\Feps^1\in L^{2}(0,T;L^{\frac{3}{2}}(\R^3))$ uniformly and $\Feps^2\in L^{\frac{4}{3}-}(0,T;L^{2}(\R^3))$ uniformly in virtue of the uniform bounds of Lemma \ref{lem:apriori} and the inequalities \eqref{ineq:energy} and \eqref{ineq:BDE}. 
As $L^{\frac{3}{2}}(\R^3)\hookrightarrow \dot{H}^{-\frac{1}{2}}(\R^3)$ we have $F_{\eps}^1\in L^{2}(0,T;\dot{H}^{-\frac{1}{2}}(\R^3))$. It follows
\begin{equation*}
    \nabla(\reps\ueps)\in L^{\frac{4}{3}-}(0,T;L^2(\R^3))+L^2(0,T;\dot{H}^{-\frac{1}{2}}(\R^3)).
\end{equation*}
We deduce that
\begin{equation*}
    \reps\ueps\in L^{\frac{4}{3}-}(0,T;\dot{H}^1(\R^3))+L^2(0,T;\dot{H}^{\frac12}(\R^3)).
\end{equation*}
Combining this bound with $\reps\ueps\in L^{4-}(0,T;L^2(\R^3))$ from Corollary \ref{coro:integrabilitym} yields
\begin{equation*}
       \reps\ueps\in L^{\frac{4}{3}-}(0,T;{H}^1(\R^3))+L^2(0,T;{H}^{\frac{1}{2}}(\R^3)).
\end{equation*}
Hence, $\reps\ueps \in L^{\frac{4}{3}-}(0,T;H^{\frac{1}{2}}(\R^3))$ as $H^{s_1}(\R^3)\subset H^{s_0}(\R^3)$ for all real $s_1\geq s_0$. Finally, using again  $\reps\ueps\in L^{4-}(0,T;L^2(\R^3))$ from Corollary \ref{coro:integrabilitym} we conclude by applying Lemma \ref{lem:interpolation} that
\begin{equation*}
    \|\reps\ueps\|_{L^p(0,T;H^s(\R^3))}\leq \|\reps\ueps\|_{L^{\frac{4}{3}-}(0,T;H^{\frac{1}{2}}(\R^3))}^{2s}\|\reps\ueps\|_{L^{{4}-}(0,T;L^2(\R^3))}^{1-2s},
\end{equation*}
for $1<p<\frac{4}{1+4s}$.
\end{proof}

\section{Acoustic waves}\label{sec:acw}

This section is devoted to the analysis of the acoustic waves in the system. For highly subsonic flows they undergo rapid oscillations in time, so that one expects the acoustic waves to converge weakly to $0$.
Furthermore, we will see that the dispersion relation satisfied by the fluctuations around the incompressible flow is not given by the classical waves but by the (scaled) Bogoliubov dispersion relation \cite{B}, which in our system reads 
\begin{equation}\label{eq:Bog_sp}
\omega(\xi)=\frac{1}{\eps}\sqrt{|\xi|^2+\eps^2\kappa^2|\xi|^4},
\end{equation}
see \eqref{op:Heps} below.\\
To perform this analysis we use identity \eqref{eq:Bohm} and rewrite system \eqref{eq:QNS} as 
\begin{equation}\label{eq:AW2}
\begin{aligned}
\begin{cases}
 \partial_t \reps+\dive (\rho_{\eps}u_{\eps})=0,\\
\partial_t (\rho_{\eps}\ueps)+\dive\left(\rho_{\eps}u_{\eps}\otimes u_{\eps}\right)+\frac{1}{\eps^2}\nabla P(\rho_{\eps})=2\nu \dive\left(\rho_{\eps}\D u_{\eps}\right)\\
\hspace{3cm} -4\kappa^2\dive\left(\nabla\sqrt{\reps}\otimes\nabla\sqrt{\reps}\right)+\kappa^2\nabla\Delta\reps,
\end{cases}
\end{aligned}
\end{equation}
where we recall that the term $\reps\D\ueps$ should be interpreted as in \eqref{eq:visc_symm}. We notice that, by using \eqref{eq:ren_ie} we can write 
\begin{equation*}
\frac{1}{\eps^2}\nabla P(\reps)=\frac{1}{\gamma\eps^2}\nabla\reps^{\gamma}=\frac{1}{\eps}\nabla\seps+(\gamma-1)\nabla\pi_{\eps},
\end{equation*}
so that upon denoting the density fluctuations $\seps=\frac{\reps-1}{\eps}$ and momentum $\meps=\reps\ueps$ equation \eqref{eq:AW2} reads
\begin{equation}\label{eq:AW3}
\begin{aligned}
\begin{cases}
&\partial_t \seps+\frac{1}{\eps}\dive (\meps)=0,\\
&\partial_t m_{\eps}+\frac{1}{\eps}\nabla\left(1-\kappa^2\eps^2\Delta\right)\seps=\Feps,
\end{cases}
\end{aligned}
\end{equation}
and
\begin{equation}\label{eq:Feps}
\Feps=\dive\left(-\reps\ueps\otimes\ueps-4\kappa^2\nabla\sqrt{\reps}\otimes\nabla\sqrt{\reps}+2\nu \rho_{\eps}\D \ueps-(\gamma-1)\pi_\eps\mathbb I\right).
\end{equation}
Projecting onto irrotational vector fields we obtain the system describing acoustic waves 
\begin{equation}\label{eq:AW}
\begin{aligned}
\begin{cases}
&\partial_t \seps+\frac{1}{\eps}\dive (\Q(\meps))=0,\\
&\partial_t \Q(m_{\eps})+\frac{1}{\eps}\nabla\left(1-\kappa^2\eps^2\Delta\right)\seps=\Q(\Feps).
\end{cases}
\end{aligned}
\end{equation}

The initial datum for \eqref{eq:AW} is given by
\begin{equation}\label{eq:dataAW1}
\sigma_{\eps}^{0}=\frac{\reps^0-1}{\eps}, \hspace{2cm} m_{\eps}^{0}=\reps^{0}u_{\eps}^{0},
\end{equation}
where we observe that from (i) of Lemma \ref{lem:fluctuations} and (iii) of Lemma \ref{lem:initialdata},
\begin{equation}\label{eq:dataAW}
\sigma_{\eps}^0\in H^{-\frac{3}{2}}(\R^3), \hspace{2cm} m_{\eps}^0\in H^{-\frac{1}{2}}(\R^3).
\end{equation}

The main result of this section shows the strong convergence to $0$ of the acoustic waves.

\begin{theorem}\label{thm: acoustic waves}
Let $(\reps, \ueps)$ be a BD-entropic weak solution of \eqref{eq:QNS}. Then, for any $0<T<\infty$, 
\begin{enumerate}[(i)]
    \item the density fluctuations $\reps-1$ converge strongly to $0$ in $C(0,T;L^{2}(\R^3))$ and in $L^{4}(0,T;H^s(\R^3))$ for any $s<1$,
    \item If $\gamma=2$, then $\seps$ converges strongly to $0$ in $L^2(0,T;L^q(\R^3))$ for any $2<q<6$,
    \item for any $2<q<\frac{12}{5}$ there exists $\delta>0$ such that 
    \begin{equation*}
        \|\Q(\meps)\|_{L^2(0,T;W^{\delta,q}(\R^3))}\leq C\eps^{\delta}.
    \end{equation*}
    \end{enumerate}
\end{theorem}

In order to infer estimates on $(\seps, \Q(\meps))$ by studying \eqref{eq:AW}, we derive Strichartz estimates for a symmetrization of the linearised system \eqref{eq:AW3} that will ultimately imply the convergence of $(\seps, \Q(\meps))$.
More precisely, we define
\begin{equation}\label{eq:defitilde}
\tilde\seps:=(1-\eps^2\kappa^2\Delta)^{\frac{1}{2}} \seps, \hspace{1.5cm} \tilde\meps:=(-\Delta)^{-\frac{1}{2}}\dive\meps,
\end{equation}
and check that if $(\seps,\meps)$ is a solution of \eqref{eq:AW3} then $(\tilde\seps,\tilde\meps)$ satisfies the symmetrised system
\begin{equation}\label{eq:symAW}
\begin{aligned}
\begin{cases}
&\partial_t\tilde\seps+\frac{1}{\eps}(-\Delta)^{\frac{1}{2}}(1-\kappa^2\eps^2\Delta)^{\frac{1}{2}}\tilde{\meps}=0,\\
&\partial_t\tilde\meps-\frac{1}{\eps}(-\Delta)^{\frac{1}{2}}(1-\kappa^2\eps^2\Delta)^{\frac{1}{2}}\tilde{\seps}=\tilde{F}_{\eps},
\end{cases}
\end{aligned}
\end{equation}
where $\tilde{F}_{\eps}=(-\Delta)^{-\frac{1}{2}}\dive F_{\eps}$. Hence, the linear evolution is characterised by the unitary semigroup $e^{-itH_{\eps}}$, where 
\begin{equation}\label{op:Heps}
H_{\eps}=\frac{1}{\eps}\sqrt{(-\Delta)(1-(\eps \kappa)^2\Delta)}
\end{equation}
is a self-adjoint operator with Fourier multiplier given by \eqref{eq:Bog_sp}.
In what follows, we are going to provide a class of Strichartz estimates for the linear propagator $e^{-itH_{\eps}}$ which will yield a control of some mixed space-time norms of $(\tilde{\seps},\tilde{\meps})$ in terms of the (scaled) Mach number.
 An interpolation argument exploiting the \emph{a priori} estimates introduced in Section \ref{sec:estimates} gives the final result. For the sake of conciseness we postpone the proof of the Strichartz estimates to the appendix \ref{app:Strichartz}. 
 
 Before stating the next Proposition, we recall that a pair of Lebesgue exponents $(p, q)$ is called Schr\"odinger admissible if $2\leq p, q\leq\infty$ and $\frac2p+\frac3q=\frac32$. Given $q\in(2,6]$, we denote
 \begin{equation}\label{eq:alpha0}
     \alpha_0=\frac{1}{2}\left(\frac{1}{2}-\frac{1}{q}\right).
 \end{equation}
 \begin{proposition}\label{prop:Strichartz}
 Let $\eps>0$, $s\in\R$ and $(p,q)$, $(p_1,q_1)$ admissible pairs with $(p,q)\neq (\infty,2)$. Then for all $\alpha\in[0,\alpha_0]$ the following estimates hold true
 \begin{equation}\label{eq:Strichartz1}
\|\eith f\|_{L^p(0,T;W^{-s-\alpha,q}(\R^3))}\leq C \eps^{\alpha}\|f\|_{H^{-s}(\R^3)},
\end{equation}
\begin{equation}\label{eq:Strichartz2}
\left\|\int_0^te^{i(t-s)\Heps}F(s)\dd s\right\|_{L^p(0,T;W^{-s-\alpha,q}(\R^3))}\leq   C \eps^{\alpha} \| F \|_{L^{p_1'}(0,T;W^{-s,q_1'}(\R^3))}.
\end{equation}
 \end{proposition}
 Proposition \ref{prop:Strichartz} will be proved in Appendix \ref{app:Strichartz}, in fact it will be a consequence of the more general Proposition \ref{prop:Strichartz final}. We notice that Proposition \ref{prop:Strichartz final} yields that \eqref{eq:Strichartz1} is valid for all $\alpha\in[0,\alpha_0]$ with $\alpha_0$ as in \eqref{eq:alpha0}. The non-homogeneous estimate \eqref{eq:Strichartz2} holds for $\alpha\in[0,\frac{d-2}{2}(1-\frac1q-\frac{1}{q_1})]$, we observe that $\alpha_0\leq\frac{d-2}{2}(1-\frac1q-\frac{1}{q_1})$ as $q_1\geq 2$.
 Let us remark that the case $\eps=1$ was already studied in \cite{GNT05}, where the authors infer dispersive estimates for the propagator $e^{itH_1}$ in order to study scattering properties for the Gross-Pitaevskii equation. In our case we need to keep track of the $\eps-$dependence of the estimates, in order to show the convergence to zero of the acoustic part. However, since $H_{\eps}=H_{\eps}(\sqrt{-\Delta})$ is a non-homogeneous function of $\sqrt{-\Delta}$, it is not possible to obtain a decay in $\eps$ by simply scaling the estimates in \cite{GNT05}. This is for example different from what happens for classical fluids \cite{DG99} where the wave-like acoustic dispersion yields the convergence to zero by scaling the estimates and by considering the fast dynamics for the fluctuations.
 
On the other hand here we can exploit that the Strichartz estimates associated to the operator \eqref{op:Heps} are sligthly better than the ones for the Schr\"odinger operator close to the Fourier origin. This fact is also noticed in \cite{GNT05} for $H_1$. By exploiting this regularizing effect, Proposition \ref{prop:Strichartz} somehow improves a class of similar estimates inferred in \cite{BDS10} in another context (the linear wave regime for the Gross-Pitaevskii equation). 
Indeed the authors of \cite{BDS10} consider $H_\eps$ in two different regimes: for low frequencies below the threshold $\frac1\eps$ the operator behaves like the wave operator, while above the threshold it is Schr\"odinger-like. In this way the low frequency part experiences a derivative loss, due to the wave-type dispersive estimates inferred.
 
 Here we do not split $H_\eps$ in low and high frequencies, nevertheless we prove the convergence to zero of the acoustic part by only losing a small amount of derivatives.
 
 In order to apply the estimates \eqref{eq:Strichartz1}, \eqref{eq:Strichartz2} to system \eqref{eq:AW}, we first need to bound $F_{\eps}$ defined in \eqref{eq:Feps} in suitable spaces.



\begin{lemma}\label{lem:boundF}
If $(\reps, u_{\eps})$ is a BD-entropic weak solution to \eqref{eq:QNS}, then one has,
\begin{enumerate}[(i)]
\item $F_{\eps}^1=\dive\left(\reps\ueps\otimes\ueps+4\kappa^2\nabla\sqrt{\reps}\otimes\nabla\sqrt{\reps}\right)+(\gamma-1)\nabla\pi_{\eps})\in L^{\infty}(0,T;W^{-2,q'}(\R^3))$ for $q'\in (1,\frac{3}{2})$,
\item
$F_{\eps}^2=2\nu \dive\left(\rho_{\eps}\D\ueps\right)\in L^{\frac{4}{3}-}(0,T; H^{-1}(\R^3))$.
\end{enumerate}
\end{lemma}

\begin{proof}
We recall that $W^{1,p}(\R^3)\hookrightarrow L^{\infty}(\R^3)$ for any $p>3$. By duality for $p'$, one has
$ L^1(\R^3)\hookrightarrow W^{-1,p'}(\R^3)$. For the first statement, we observe that 
\begin{equation*}
\left(\reps\ueps\otimes\ueps\right)+4\kappa^2\left(\nabla\sqrt{\reps}\otimes\nabla\sqrt{\reps}\right)+(\gamma-1)\pi_{\eps}\mathbb I\in L^{\infty}(0,T;L^1(\R^3)),
\end{equation*}
and thus $F_{\eps}^1\in L^{\infty}(0,T; W^{-2,q'}(\R^3))$ for $q\in (3,\infty)$. Regarding the second statement, we observe that 
\begin{equation*}
\begin{aligned}
&\|\sqrt{\reps}\Seps\|_{L^{\frac{4}{3}-}(0,T;L^2(\R^3))} \\
&\leq C_T\left( \|\Seps\|_{L^2(0,T;L^2(\R^3))}+\|\sqrt{\reps} -1\|_{L^{4-}(0,T;L^{\infty}(\R^3))}\|\Seps\|_{L^2(0,T;L^2(\R^3))}\right),
\end{aligned}
\end{equation*}
and thus $F_{\eps}^2\in L^{\frac{4}{3}-}(0,T; H^{-1}(\R^3))$.
\end{proof}

\begin{remark}
Here,  we need to use Strichartz estimates in non-homogeneous spaces. This is due to the fact that $L^1$ fails to embed in a homogeneous Sobolev space.
\end{remark}

By combining the dispersive estimates of Proposition \ref{prop:Strichartz} and the bounds in Lemma \ref{lem:boundF} we can then infer the convergence to zero of 
$(\sigma_{\eps},\Q(\meps))$.

\begin{proposition}\label{prop: StrichartzBesov}
Let $(\seps,\meps)$ be solution of \eqref{eq:AW3} with initial data $(\sigma_{\eps}^0,m_{\eps}^0)$. For any $s\in\R$, $(p,q)$ admissible pair with $q>2$ and $\alpha_0$ as in \eqref{eq:alpha0}, one has for any $\alpha\in[0,\alpha_0]$ and any admissible pairs $(p_1,q_1)$, $(p_2,q_2)$ that
\begin{equation}\label{eq:final estimate}
\begin{aligned}
&\|(\seps,\Q(\meps))\|_{L^p(0,T;{W}^{-s-\alpha,q}(\R^3))}\\
&\leq C_T \eps^{\alpha} \Big( \|\sigma_{\eps}^0\|_{H^{-s}(\R^3)}+\|\eps\nabla\sigma_{\eps}^0\|_{H^{-s}(\R^3)}+\|m_{\eps}^0\|_{H^{-s}(\R^3)}\\
&+\left\|F_{\eps}^1\right\|_{L^{p_1'}(0,T;W^{-s,q_1'}(\R^3))}+\left\|F_{\eps}^2\right\|_{L^{p_2'}(0,T;W^{-s,q_2'}(\R^3))}\Big).
\end{aligned}
\end{equation}
In particular, if $(\reps,\ueps)$ is a BD-entropic solution to \eqref{eq:QNS} with initial data $(\reps^0,\ueps^0)$, $s\geq 2$ and $(p,q)$  any admissible pair with $q>2$ then
\begin{equation}\label{eq: convergence rate}
\|(\seps,\Q(\meps))\|_{L^p(0,T;W^{-s-\alpha,q}(\R^3)}\leq C_T \eps^{\alpha},
\end{equation}
for all $\alpha\in [0,\alpha_0]$.
\end{proposition}

The condition $s\geq 2$ is due to the low regularity of the nonlinearity in \eqref{eq:AW}.

\begin{proof}
First we observe that for any $s\in \R$, $q\in(1,\infty)$ and for $\Tilde{\seps},\Tilde{\meps}$ as defined in \eqref{eq:defitilde} that
\begin{equation}\label{eq:claim1}
\|\seps\|_{L^p(0,T;W^{s,q}(\R^3))}\leq C \|\tilde\seps\|_{L^p(0,T;W^{s,q}(\R^3))},
\end{equation}
and 
\begin{equation}\label{eq:claim2}
\|\Q(\meps)\|_{L^p(0,T;W^{s,q}(\R^3))}\leq C \|\tilde\meps\|_{L^p(0,T;W^{s,q}(\R^3))}.
\end{equation}
Indeed, to show \eqref{eq:claim1} we define  $T(f)=(1-\eps^2\kappa^2\Delta)^{-\frac{1}{2}}f$. The symbol of $T$ is given by $m(\xi)=(1+\eps^2\kappa^2|\xi|^2)^{-\frac{1}{2}}$. It is straight forward to check that $T$  is a pseudo-differential operator of order $0$ and thus $T:W^{s,q}\rightarrow W^{s,q}$ bounded. The inequality \eqref{eq:claim2}, follows from observing that the projection on the gradient part $\Q$ is given by a matrix valued Fourier multiplier $m(\xi)=\frac{\xi_k \xi_j}{|\xi|^2}$ while the change of variables  $(-\Delta)^{-\frac{1}{2}}\dive$ corresponds to the multiplier $\frac{\xi_j}{|\xi|}$. Inequality \eqref{eq:claim2} follows. Second, we notice that \eqref{eq:claim1}, \eqref{eq:claim2} combined with the Strichartz estimates \eqref{eq:Strichartz1} and \eqref{eq:Strichartz2} yield that for any $(p,q)$ admissible with $q>2$ there exists $\alpha_0>0$ such that for any $\alpha\in[0,\alpha_0]$ one has
\begin{equation}\label{eq:final estimate1}
\begin{aligned}
&\|(\seps,\Q(\meps))\|_{L^p(0,T;{W}^{-s-\alpha,q}(\R^3))}\leq \|(\tilde\seps,\tilde\meps)\|_{L^p(0,T;{W}^{-s-\alpha,q}(\R^3))}\\
&\leq C_T \eps^{\alpha} \left( \|(\tilde{\sigma}_{\eps}^0,\tilde{m}_{\eps}^0)\|_{H^{-s}(\R^3)}+\left\|\tilde{F}_{\eps}^1\right\|_{L^{p_1'}(0,T;W^{-s,q_1'}(\R^3))}+\left\|\tilde{F}_{\eps}^2\right\|_{L^{p_2'}(0,T;W^{-s,q_2'}(\R^3))}\right),
\end{aligned}
\end{equation}
provided that $(p_1,q_1), (p_2,q_2)$ are admissible. Finally, we have that 
\begin{equation}\label{eq:mepstilde}
    \|\Tilde{m}_{\eps}^0\|_{H^{-s}(\R^3)}\leq \|m_{\eps}^0\|_{H^{-s}(\R^3)}
\end{equation}
as the operator $m_{\eps}^0\mapsto \tilde{m}_{\eps}^0$ with symbol $m(\xi)=\frac{\xi_i}{|\xi|}$ is of order $0$. Analogously, one derives the respective bounds for $\Tilde{F}_{\eps}^i$ with $i=1,2$. The operator $T^{-1}: {\seps}^0\mapsto(1-\eps^2\kappa^2\Delta)^{\frac12}\seps^0$ is characterised by the symbol $m(\xi)=(1+\eps\kappa^2|\xi|^2)^{\frac{1}{2}}$. One has that
\begin{equation*}
    m(\xi)\sim\begin{cases}
    C \qquad &|\xi|\leq \frac{1}{\eps},\\
    \eps|\xi| \qquad &|\xi|\geq \frac{1}{\eps}.
    \end{cases}
\end{equation*}
It follows,
\begin{equation}\label{eq:sepstilde}
\begin{aligned}
\|\tilde \sigma_{\eps}^0\|_{H^{-s}(\R^3)}&=\|(1-\eps^2\kappa^2\Delta)^{\frac12}\seps^0\|_{H^{-s}(\R^3)}\\
&\leq C\left( \| P_{\leq\frac{1}{\eps}}(\sigma_{\eps}^0)\|_{H^{-s}(\R^3)}+ \| P_{>\frac{1}{\eps}}(\eps\nabla\sigma_{\eps}^0)\|_{H^{-s}(\R^3)}\right).
\end{aligned}
\end{equation}
Inequality \eqref{eq:final estimate} now follows from \eqref{eq:final estimate1} combined with \eqref{eq:mepstilde} and \eqref{eq:sepstilde}. 
It remains to show \eqref{eq: convergence rate}. We recall that $\sigma_{\eps}^0\in H^{-\frac{3}{2}}(\R^3)$ and $\eps\nabla\sigma_{\eps}^0\in H^{-\frac{1}{2}}(\R^3)$ uniformly in $\eps$ in virtue of Lemma \ref{lem:fluctuations}. Further, $\meps^0\in H^{-\frac{1}{2}}(\R^3)$ from Lemma \ref{lem:initialdata}. Finally, for $(p_1,q_1)=(\frac{8}{3},4)$ and $(p_2,q_2)=(\infty, 2)$ we have in virtue of Lemma \ref{lem:boundF} that
\begin{equation}\label{eq:FtildeF}
\begin{aligned}
    \|\tilde{F}_{\eps}^1\|_{L^{p_1'}(0,T;W^{-2,q_1'}(\R^3))} &\leq C_T \|{F}_{\eps}^1\|_{L^{\infty}(0,T;W^{-2,q_1'}(\R^3))},\\
    \|\tilde{F}_{\eps}^2\|_{L^{p_2'}(0,T;H^{-1}(\R^3))} &\leq C_T \|{F}_{\eps}^2\|_{L^{\frac{4}{3}-}(0,T;H^{-1}(\R^3))}.\\
    \end{aligned}
\end{equation}
Equation \eqref{eq:final estimate} now follows from \eqref{eq:final estimate1} and the uniform bounds for the terms in the parenthesis on the right-hand side provided that $s\geq 2$.  This completes the proof.
\end{proof}

\subsection*{Proof of Theorem \ref{thm: acoustic waves}.}

\begin{proof}
Lemma \ref{lem:apriori} (i) states that $\reps-1$ converges strongly to $0$ in $L^{\infty}(0,T;L^2(\R^3))$ with explicit convergence rate. Statement (iv) of Lemma \ref{lem:fluctuations} yields that $\reps-1=\eps\seps\in L^{4}(0,T;H^1(\R^3))$ uniformly bounded. Interpolation of these bounds gives the desired strong convergence in $L^4(0,T;H^s(\R^3))$ for $s\in[0,1)$. \\
Statement (ii) is inferred by observing that from (v) of Lemma \ref{lem:fluctuations} one has $\seps\in L^2(0,T;H^1(\R^3))$ uniformly if $\gamma=2$. Provided that $q\in[2,6)$, one has $s=1-3(\frac12-\frac1q)>0$ and $\seps\in L^2(0,T;W^{s,q}(\R^3))$ uniformly for $\gamma=2$. Inequality \eqref{eq: convergence rate} yields that $\seps$ converges to $0$ in $L^p(0,T;W^{-2-\alpha,q}(\R^3))$ for $(p,q)$ admissible with $q\in (2,6)$ and $\alpha>0$ sufficiently small with convergence rate $\eps^{\alpha}$. We obtain by applying Lemma \ref{lem:interpolation} that
\begin{equation*}
    \|\seps\|_{L^2(0,T;W^{s_0,q}(\R^3))}\leq C_T\|\seps\|_{L^p(0,T;W^{-2-\alpha,q}(\R^3))}^{\theta}\|\seps\|_{L^2(0,T;W^{s,q}(\R^3))}^{1-\theta},
\end{equation*}
where $\theta\in(0,1)$ and 
\begin{equation*}
   s=1-3\left(\frac{1}{2}-\frac{1}{q}\right), \qquad 
    s_0=\theta(-2-\alpha)+(1-\theta)s, \qquad 
    \frac{1}{2}>\frac{\theta}{p}+\frac{(1-\theta)}{2}.
\end{equation*}
We notice that for $(p,q)$ admissible such that $q\in(2,6)$ one has $p\in (2,\infty)$ and $s>0$. 
Hence, for any admissible pair $(p,q)$ such that $q\in(2,6)$ there exists $s_0>0$ such that $\seps$ converges strongly to $0$ in $L^p(0,T;W^{s_0,q}(\R^3))$. In particular, $\seps$ converges strongly to $0$ in $L^2(0,T;L^q(\R^3))$ for any $q\in(2,6)$ at convergence rate $\eps^{\alpha\theta}$. \\
Finally, to obtain a bound on $\Q(\meps)$, we interpolate between the \emph{a priori} bound \eqref{eq:boundm} and the inequality \eqref{eq: convergence rate}. Proposition \ref{prop: StrichartzBesov} yields that for $q>2$ such that $(p,q)$ is admissible there exists $\alpha_0=\alpha_0(q)>0$ such that
\begin{equation*}
    \|\Q(\meps)\|_{L^p(0,T;W^{-2-\alpha,q}(\R^3))}\leq C\eps^{\alpha},
\end{equation*}
for any $\alpha\in [0,\alpha_0]$. We notice that \eqref{eq:boundm} implies that
\begin{equation*}
    \meps\in L^{2-}(0,T;H^{\frac{1}{4}}(\R^3))
\end{equation*}
uniformly in $\eps$.
By Sobolev embedding, it follows for $q\in(2,\infty)$ that 
\begin{equation}\label{eq:Qmeps}
    \|\Q(\meps)\|_{L^{2-}(0,T;W^{s_1,q}(\R^3))}\leq C,
\end{equation}
with $s_1=\frac{1}{4}-3(\frac{1}{2}-\frac{1}{q})$. We notice that $s_1>0$ provided that $q\in [2,\frac{12}{5})$. By applying Lemma \ref{lem:interpolation} to interpolate between \eqref{eq: convergence rate} and \eqref{eq:Qmeps}, we obtain that for $q\in(2,\frac{12}{5})$ with $(p,q)$ admissible
\begin{equation*}
    \|\Q(\meps)\|_{L^{p_0}(0,T;W^{s_0,q}(\R^3))}\leq  \|\Q(\meps)\|_{L^{p}(0,T;W^{-2-\alpha,q}(\R^3))}^{\theta}\|\Q(\meps)\|_{L^{2-}(0,T;W^{s_1,q}(\R^3))}^{1-\theta},
\end{equation*}
where $(\theta,p_0,s_0)$ are such that
\begin{equation*}
    0<\theta<1,\qquad 
    s_0=\theta(-2-\alpha)+(1-\theta)s_1, \qquad 
    \frac{1}{p_0}=\frac{\theta}{p}+\frac{1-\theta}{2-}.
\end{equation*}
We observe that provided $s_1>0$ there exists $\theta\in(0,1)$ such that $s_0>0$. As $(p,q)$ is an admissible pair with $q\in (2,\frac{12}{5})$ it follows in particular that $p\in(8,\infty)$. 
Hence, $p_0\geq 2$ for all $\theta\in(0,1)$. Therefore, for any $(p,q)$ admissible with $q\in(2,\frac{12}{5})$ there exist $\alpha\in[0,\alpha_0]$, $\theta\in(0,1)$, $s_0>0$ and $p_0\geq 2$ such that 
\begin{equation*}
    \|\Q(\meps)\|_{L^{p_0}(0,T;W^{s_0,q}(\R^3))}\leq C\eps^{\alpha\theta},
\end{equation*}
upon applying \eqref{eq: convergence rate}. The final statement follows by choosing $\delta=\min(s_0,\alpha\theta)$.
\end{proof}


\section{Convergence to the limiting system}\label{sec:conv}
In this section, we show strong compactness of $\sqrt{\reps}\ueps$ in $L^2(0,T;L_{loc}^2(\R^3))$ and conclude the proof of Theorem \ref{thm:main}. To that end, we first prove strong compactness of the incompressible part $\P(\reps\ueps)$ that together with Theorem \ref{thm: acoustic waves} will yield strong compactness of $\reps\ueps$ in $L^2(0,T;L_{loc}^2(\R^3))$. We notice that if a subsequence $\sqrt{\reps}\ueps$ converges weakly to some $u$ in $L^{\infty}(0,T;L^2(\R^3))$ then $\reps\ueps$ converges weakly to $u$ in $L^2(0,T;L^2(\R^3))$. Indeed, $\reps\ueps=\sqrt{\reps}\ueps+(\sqrt{\reps}-1)\sqrt{\reps}\ueps$ where the second summand  converges weakly to $0$ as $\sqrt{\reps}\ueps$ uniformly bounded and $\sqrt{\reps}-1$ converges strongly to $0$ in $L^{\frac83}(0,T;L^{\infty}(\R^3))$ in virtue (v) of Lemma \ref{lem:apriori}.

\begin{lemma}\label{prop: convergence P}
Under the assumptions of Theorem \ref{thm:main},  let $\reps\ueps$ be a subsequence converging weakly to some $u\in L^{2}(0,T;L^2(\R^3))$. Then $\reps\ueps$ converges strongly to $u$ in $L^{2}(0,T;L_{loc}^2(\R^3))$ as $\eps$ goes to $0$.
\end{lemma}

\begin{proof}
 We decompose $\rho_\eps u_\eps=\Q(\reps\ueps)+\P(\reps\ueps)$ by means of the Leray-Helmholtz projection operator. Theorem \ref{thm: acoustic waves} states that $\Q(\reps\ueps)$ converges strongly to $0$ in $L^2(0,T;W^{\delta,q}(\R^3))$ for any $q\in(2,\frac{12}{5})$ and $\delta=\delta(q)>0$ sufficiently small. As $W^{\delta,q}(\R^3)$ is continuously embedded in $L^q(\R^3)$, it follows that $\Q(\reps\ueps)$ converges strongly to $0$ in $L^2(0,T;L^q(\R^3))$ for  $q\in(2,\frac{12}{5})$ and therefore in particular in $L^{2}(0,T;L_{loc}^2(\R^3))$. It remains to analyse the convergence of the incompressible part $\P(\reps\ueps)$. From \eqref{eq:boundm}, we have $\P(\reps\ueps)\in L^{p}(0,T;H^s(\R^3))$ for $0\leq s\leq \frac{1}{2}$ and $1\leq p<\frac{4}{1+4s}$. Moreover, from 
\begin{equation*}
 \partial_t \P(\reps\ueps)=-\P\left(\dive\left(\reps\ueps \otimes u_{\eps}\right)\right)+2\nu \P\left(\dive(\sqrt{\rho_{\eps}}\Seps)\right)+\kappa^2 \P\left(\dive\left(\nabla\sqrt{\reps} \otimes \nabla\sqrt{\reps} \right)\right),
\end{equation*}
with $\Seps$ defined in \eqref{eq:visc_symm}, 
we conclude that $\partial_t \P(\reps\ueps)\in L^{2}(0,T;H^{-s}(\R^3)$ for any $s>\frac{5}{2}$. Indeed, it suffices to observe that from the energy bounds of Lemma \ref{lem:apriori} we have $\nabla\sqrt{\reps}\in L^{\infty}(0,T;L^2(\R^3))$, $\sqrt{\reps}-1\in L^{\infty}(0,T;L^2(\R^3))$, $\Seps\in L^2(0,T;L^2(\R^3))$ and $\sqrt{\reps}\ueps\in L^{\infty}(0,T;L^2(\R^3))$. As the embedding $H^s(\R^3)\hookrightarrow L_{loc}^2(\R^3)$ is compact for any $s>0$, the Aubin-Lions Lemma yields strong compactness of $\P(\reps\ueps)$ in $L^2(0,T;L_{loc}^2(\R^3))$. It follows that 
\begin{equation*}
    \reps\ueps\rightarrow u \qquad \text{strongly in} \quad L^2(0,T;L_{loc}^2(\R^3)).
\end{equation*}
\end{proof}
By combing the strong compactness of $\reps\ueps$ and the strong convergence of $\sqrt{\reps}-1$ to $0$, we infer strong compactness of $\sqrt{\reps}\ueps$ and pass to the limit in \eqref{eq:QNS}.
\begin{proposition}\label{prop: convergence}
Under the assumption of Theorem \ref{thm:main}, let $\sqrt{\reps}\ueps$ be a subsequence weakly converging to some $u\in L^{\infty}(0,T;L^2(\R^3))$. Then $\sqrt{\reps}\ueps$ converges strongly to $u$ in $L^2(0,T;L_{loc}^2(\R^3))$ and $u$ is a global weak solution to the incompressible Navier-Stokes equation with initial data $u\big|_{t=0}=\P(u_0)$ with $u^0$ defined in \eqref{eq: ID}.
\end{proposition}
We notice that the passage to the limit relies on the additional uniform bounds provided by the BD entropy inequality \eqref{ineq:BDE} in a crucial way: for both the strong compactness of $\sqrt{\reps}\ueps$ and the convergence to $0$ of the dispersive tensor.
\begin{proof}
First, we show strong compactness of $\sqrt{\reps}\ueps$. We notice that
\begin{equation*}
\sqrt{\reps}\ueps=\reps\ueps+(1-\sqrt{\reps})\sqrt{\reps}\ueps.
\end{equation*}
 For any compact $K\subset \R^n$, one has
\begin{equation*}
\begin{aligned}
\|\sqrt{\reps}&\ueps-u\|_{L^2(0,T;L^2(K))}\leq \|\reps\ueps-u\|_{L^2(0,T;L^2(K))}+\|(1-\sqrt{\reps})\sqrt{\reps}\ueps\|_{L^2(0,T;L^2(K))}\\
&\leq \|\reps\ueps-u\|_{L^2(0,T;L^2(K))}
+C \|(1-\sqrt{\reps})\|_{L^{2}(0,T;L^{\infty}(K))}\|\sqrt{\reps}\ueps\|_{L^{\infty}(0,T;L^2(K))}\\
&\leq C \left( \|\reps\ueps-u\|_{L^2(0,T;L^2(K))}\|+ \eps^{\beta}\right),
\end{aligned}
\end{equation*}
for some $\beta>0$, where we used the convergence provided by (v) of Lemma \ref{lem:apriori} in the last step. Second, we carry out the $\eps$-limit in the weak formulations of \eqref{eq:QNS}. The strong compactness of $\reps\ueps$ provided by Lemma \ref{prop: convergence P} and \eqref{eq:alpha} applied to $\reps$ and $\reps^0$ allow us to pass to the limit in the weak formulation of the continuity equation. We recover,
\begin{equation*}
    \diver u=0 \qquad \text{in} \quad \mathcal{D}'\left((0,T)\times \R^3 \right).
\end{equation*}
Next, we observe that the dispersive tensor satisfies
\begin{equation*}
\kappa^2\nabla\Delta\reps-4\kappa^2\dive\left(\nabla\sqrt{\reps}\otimes\nabla\sqrt{\reps}\right)\rightarrow 0 \qquad \text{in} \quad \mathcal{D}'([0,T)\times\R^3).
\end{equation*}
We notice that $\nabla\sqrt{\reps}$ converges strongly to $0$ in $L^2(0,T;L^2(\R^3))$. Indeed, one has that $\nabla^2\sqrt{\reps}\in L^2(0,T;L^2(\R^3))$ uniformly bounded stemming from the BD entropy inequality \eqref{ineq:BDE}. It follows, that $\sqrt{\reps}-1$ converges strongly to $0$ in $L^2(0,T;H^1(\R^3))$, see (v) of Lemma \ref{lem:apriori}. Similarly, $\reps-1$ converges strongly to $0$ in $L^{\infty}(0,T;L^2(\R^3))$ from (i) of Lemma \ref{lem:initialdata}. Finally, we consider the weak formulation of the momentum equation projected onto divergence free vector fields. Let $\psi\in C_c^{\infty}([0,T)\times \R^3;\R^3)$ such that $\dive\psi=0$, then the momentum equation reduces to
\begin{equation}
\begin{aligned}\label{eq: weak formulation}
&
\int_{\R^3}\rho_{\eps}^0u_{\eps}^0\psi(0)+\int_0^T\int_{\R^3}\sqrt{\reps}\sqrt{\reps}\ueps \psi_t+(\sqrt{\reps}\ueps\otimes\sqrt{\reps}\ueps)\nabla\psi\\
&-2\nu \int_0^T\int_{\R^3}\left(\sqrt{\reps}\ueps\otimes \nabla \sqrt{\reps}\right)\nabla \psi-2\nu \int_0^T\int_{\R^3}\left(\nabla\sqrt{\reps}\otimes \sqrt{\reps}\ueps\right)\nabla \psi\\
&+\nu \int_0^T\int_{\R^3}\sqrt{\reps}\sqrt{\reps}\ueps\Delta\psi
-4\kappa^2 \int_0^T\int_{\R^3}\left(\nabla\sqrt{\reps}\otimes \nabla\sqrt{\reps}\right)\nabla\psi=0,
\end{aligned}
\end{equation}
The strong convergence of $\sqrt{\reps}\ueps$ in $L^2(0,T;L_{loc}^2(\R^3))$ together with Lemma \ref{lem:apriori} is sufficient to pass to the limit in \eqref{eq: weak formulation}. We conclude that \eqref{eq: weak formulation} converges to
\begin{equation*}
\int_{\R^3}\P(u_{0})\psi(0)+\int_0^T\int_{\R^3} u  \psi_t+\left(u \otimes u\right)\nabla\psi
+\nu \int_0^T\int_{\R^3}u \Delta\psi=0,
\end{equation*}
where we used that $\rho_{\eps}^0u_{\eps}^0$ converges weakly to $u_0$ in $L_{loc}^{\frac{3}{2}-}(\R^3)$ as consequence of \eqref{eq: ID} and Lemma \ref{lem:initialdata}. Indeed, (ii) of Lemma \ref{lem:initialdata} implies that $\sqrt{\reps^0}-1$ converges strongly to $0$ in $L^{\infty}(0,T;L^q(\R^3))$ for any $q\in [2,6)$ by interpolation. We conclude by exploiting that $\psi$ is divergence free. 
Therefore, there exists a distribution $p$ defined on $(0,T)\times \R^3$ such that $u$ is solution of
\begin{equation*}
\partial_{t}u+u\cdot \nabla u+\nabla p=2\nu\Delta u, \quad \dive u=0 \quad \text{in} \, \, \mathcal{D}'\left((0,T)\times \R^3 \right),
\end{equation*}
with initial data $\P(u_0)$.
\end{proof}

As we already said, at fixed $\eps>0$ the finite energy weak solutions $(\rho_\eps, u_\eps)$ to \eqref{eq:QNS} satisfy a weak version of the energy inequality due to the degenerate viscosity, namely
\begin{equation*}
E(t)+2\nu\int_0^t\left|\Seps\right|^2\text{d}s\text{d}x\leq E(0),
\end{equation*}
where $\Seps$ is given by \eqref{eq:visc_symm}. We remark that in fact in the limit as $\eps\to0$ it is possible to recover the usual energy dissipation. More precisely, the uniform boundedness of $\Seps\in L^2(0,T,L^2(\R^3))$ only yields that $\Seps\rightharpoonup {S}$ weakly in $L^2((0,T)\times \R^3)$ up to subsequences. In the next Proposition we show that in fact we have  ${S}=\frac{1}{2}\D u$. Moreover, by assuming the initial data to be well-prepared we obtain the convergence of the total energy at initial time and thus we can also show that the limit function $u$ obtained is indeed a Leray solution.

\begin{proposition}\label{prop:strong energyinequality}
Under the assumptions of Theorem \ref{thm:main},  let $\Seps$ be as defined in \eqref{eq:visc_symm}  then
\begin{equation*}
\Seps\rightharpoonup \D u \qquad \text{in} \quad L^2((0,T)\times \R^3).
\end{equation*}
Consequently, $u$ is a weak solution to \eqref{eq:INS} that satisfies $u\in L^{\infty}(0,T;L^2(\R^3))\cap L^{2}(0,T;\dot{H}^1(\R^3))$.
If additionally,  $(\reps^0,\ueps^0)$ satisfies \eqref{eq:wellprepared}, then $u$ is a Leray solution of \eqref{eq:INS}, i.e. it satisfies \eqref{ineq:Leray}.
\end{proposition}

\begin{proof}
In virtue of Proposition \ref{prop: convergence}, $u$ is a weak solution of \eqref{eq:INS} with initial data  $u\big|_{t=0}=\P(u_0)$. Next, we show that $\Seps\rightharpoonup \D u$ in $L^2(0,T;L^2(\R^3))$.
From \eqref{ineq:energy}, one has that there exists $S\in L^2((0,T;L^2(\R^3))$ such that $\Seps\rightharpoonup S$ weakly in $L^2((0,T)\times \R^3)$ up to passing to subsequences. Moreover, $\sqrt{\reps}\Seps\rightarrow S$ in $\mathcal{D}'((0,T)\times\R^3)$. Indeed, let us write $\sqrt{\reps}\Seps=\Seps+(\sqrt{\reps}-1)\Seps$. The second term converges to $0$ in $\mathcal{D}'((0,T)\times\R^3)$ since $\sqrt{\reps}-1\rightarrow 0$ strongly in $L^{\infty}(0,T,L^q(\R^3))$ for $2\leq q<6$ from Lemma \ref{lem:apriori}. On the other hand, from \eqref{eq:visc_symm} and \eqref{eq:visc_tens} we infer that $\sqrt{\reps}\Seps\rightarrow \D u$ in $\mathcal{D}'((0,T)\times\R^3)$. Indeed, from Proposition \ref{prop: convergence P}, we have $\nabla(\reps\ueps)\rightarrow \nabla u$ in $\mathcal{D}'((0,T)\times\R^3)$ and from  $\nabla\sqrt{\reps}\rightarrow 0$ in $L^{2}(0,T;L^2(\R^3))$ by Lemma \ref{lem:apriori}, it follows $\sqrt{\reps}\ueps\otimes\nabla\sqrt{\reps}\rightarrow 0$ in $L^{2}(0,T,L^1(\R^3))$.  Thus $S=\D u\in L^2(0,T;L^2(\R^3))$. We observe that for $u\in H^1(\R^3)$ such that $\dive u =0$, one has
\begin{equation*}
\int_{\R^3}\left|\nabla u\right|^2\dd x=2\int_{\R^3}\left|\D u\right|^2\dd x.
\end{equation*}
Finally, by lower semi-continuity we conclude that 
\begin{align*}
&\int_{\R^3}\frac{1}{2}|u|^2\dd x+\nu \int_{0}^t\int_{\R^3}\left|\nabla u\right|^2\dd x \dd t\\
&\leq \liminf_{\eps\rightarrow 0} \int_{\R^3}\frac{1}{2}\reps|\ueps|^2+\kappa^2|\nabla\sqrt{\reps}|^2\dd x+2\nu \int_{0}^t\int_{\R^3}\left|\Seps\right|^2\dd x \dd t\\
& \leq \int_{\R^3}\frac{1}{2}\reps^0|\ueps^0|^2+\kappa^2|\nabla\sqrt{\reps^0}|^2+\pi_{\eps}(\reps^0)\dd x.
\end{align*}
Thus, $u\in L^{\infty}(0,T;L^2(\R^3))$  and $\nabla u\in L^{2}(0,T;L^2(\R^3))$. 
In order to conclude \eqref{ineq:Leray}, it remains to show that,
\begin{equation*}
    \int_{\R^3}\frac{1}{2}\reps^0|\ueps^0|^2+\kappa^2|\nabla\sqrt{\reps^0}|^2+\pi_{\eps}(\reps^0) \dd x\rightarrow \int_{\R^3}\frac{1}{2}|u^0|^2\dd x.
\end{equation*}
If the initial data is well-prepared, namely satisfies \eqref{eq:wellprepared}, the proof is complete. 
\end{proof}

Finally, we stress that, since the bounds obtained in Proposition \ref{prop:boundm} are uniform in $\eps>0$, they are also inherited by the solution to \eqref{eq:INS} obtained in the limit. The next Proposition proves the last statement of Theorem \ref{thm:main}.
\begin{proposition}\label{prop:boundu}
Let $u$ be the solution to \eqref{eq:INS} obtained in the limit. Then for any $0<T<\infty$ one has $u\in L^p(0, T;H^s(\R^3))$ with $0\leq s\leq\frac12$ and $1\leq p<\frac{1}{1+4s}$.
\end{proposition}
\section*{Acknowledgments}
The first and the third author acknowledge partial support by PRIN-MIUR project 2015YCJY3A\_003 \emph{Hyperbolic Systems of Conservation Laws and Fluid Dynamics: Analysis and Applications}. The first and second author acknowledge partial support through the INdAM-GNAMPA project \emph{Esistenza, limiti singolari e comportamento asintotico per equazioni Eulero/Navier--Stokes--Korteweg}. 

\appendix

\section{Energy and BD entropy inequality}\label{app:energy}
In this Section we discuss the existence of BD-entropic weak solutions to \eqref{eq:QNS} fulfilling the hypotheses of our main Theorem \ref{thm:main}. More specifically, we focus on the existence of weak solutions satisfying the energy inequality \eqref{ineq:energy} and the BD entropy inequality \eqref{ineq:BDE}.
Our specific interest stems from the fact that, for well-prepared initial data we  want to show the convergence towards Leray weak solutions, see the Proposition \ref{prop:leray}, namely we want to recover the inequality \eqref{ineq:Leray} in the limit. For this purpose we need the weak solutions to \eqref{eq:QNS} to satisfy the energy inequality in the form \eqref{ineq:energy}.\\
The QNS system has already been studied in the literature, see for example \cite{AS, LV16}. however both papers consider the problem in the periodic domain $\T^d$ and more importantly the energy inequality satisfied there slightly differs from \eqref{ineq:energy}, see the Definition 1 in \cite{AS} or Theorem 1.1 in \cite{LV16}.\\
Let us emphasize that in this Appendix too, we focus on the periodic domain. Indeed, once \eqref{ineq:energy} and \eqref{ineq:BDE} are proved on $\T^d$, then by exploiting an invading domain approach the same result can be proved also on $\R^d$. We address the interested reader to the recent preprint \cite{AHS20}, where the Cauchy problem for the QNS system is studied with non-trivial far-field behavior. In \cite{AHS20} it is shown that, if we have a sequence of weak solutions satisfying \eqref{ineq:energy} and \eqref{ineq:BDE} on larger and larger domains, then by a suitable truncation argument, it is possible to construct a weak solution satisfying the same entropy inequalities on $\R^d$.\\
As already said, one of the main problems in studying weak solutions to \eqref{eq:QNS1} is the lack of control for the velocity field and its gradient in the vacuum region. This fact immediately implies that in general weak solutions to \eqref{eq:QNS1} do not satisfy \eqref{ineq:energy} and \eqref{ineq:BDE}. Furthermore, it is still an open problem to determine the minimal assumptions for weak solutions such that \eqref{ineq:energy} and \eqref{ineq:BDE} are valid. We address the interested reader to \cite{NNT} for a recent result in this direction. 
On the other hand, by using a suitable approximation procedure, it is possible to show the existence of weak solutions to \eqref{eq:QNS1} for which both \eqref{ineq:energy} and \eqref{ineq:BDE} are valid. This indeed reflects the commonly adopted strategy of showing the existence of weak solutions, which consists in constructing a sequence of (more regular) approximating solutions that also satisfy approximate versions of \eqref{ineq:energy} and \eqref{ineq:BDE}. By passing to the limit and by using the compactness provided by the a priori bounds, it is then possible to find a weak solution to the original system which additionally satisfies the estimates \eqref{ineq:energy} and \eqref{ineq:BDE}.\\
Let us stress that this strategy does not allow in any case to recover the energy inequality in the form \eqref{eq:en_diss}. Indeed, while the a priori bounds imply the compactness of $\sqrt{\rho_n}u_n$ - which allows to pass to the limit in the convective term - the only information we have on the gradient of the velocity field is that $\rho_n|\nabla u_n|^2$ is uniformly bounded in $L^1$. Thus, in the limit we can only infer $\sqrt{\rho_n}\nabla u_n\rightharpoonup\mathbf T$ in $L^2$, where $\mathbf T$ is the tensor determined by the identity in \eqref{eq:visc_tens}.\\
In this Appendix we follow the approximation argument adopted in \cite{AS} to show that the weak solutions constructed in \cite{AS} indeed satisfy \eqref{ineq:energy} and \eqref{ineq:BDE}. In the following, we consider \eqref{eq:QNS}  on $[0,T)\times\T^d$. All domains of integration are hence adapted to $\T^d$ and the energy functional reads
\begin{equation*}
E(\rho,u)=\int_{\T^d}\frac12{\rho}|u|^2+2\kappa^2|\nabla\sqrt{\rho}|^2+\frac1\gamma\rho^{\gamma}\dd x.
\end{equation*}
In particular, the internal energy does not need to be renormalized by considering $\pi(\rho)$ as in \eqref{eq:en}. We refer the reader to \cite{AHS20} for a more detailed discussion. 
\begin{theorem}\label{thm:A1}
Let $d=2,3$. Let $\nu, \kappa$ and $\gamma$ positive such that $\kappa<\nu$, $\gamma>1$ for $d=2$ and   $\kappa^2<\nu^2<\frac{9}{8}\kappa^2$ and $1<\gamma<3$ for $d=3$. For any $0<T<\infty$ and initial data $(\rho^0,u^0)$ such that $E(\rho^0,u^0)<+\infty$ and in addition $u^0=0$ on $\{\rho^0=0\}$, $\sqrt{\rho^0}u^0\in L^2(\T^d)\cap L^{2+}(\T^d)$, there exists a BD-entropic weak solution $(\rho,u)$ of \eqref{eq:QNS1} on $\T^d$ with initial data $(\rho^0,u^0)$.
\end{theorem}
We stress that existence has been introduced in \cite{AS}, here we show the validity of \eqref{ineq:energy} and \eqref{ineq:BDE}. The weak solutions constructed in \cite{AS} are obtained as limit of a sequence of approximating solutions $\{(\rd, \ud )\}_{\delta}$ satisfying the following system. 
\begin{multline}\label{eq:app}
\begin{cases}
\partial_t\rd+\dive (\rd\ud)=0\\
\partial_t(\rd\ud)+\dive(\rd\ud\otimes\ud)+\nabla\left((\rd)^{\gamma}+P_{\delta}(\rd)\right)+\tilde{p}_{\delta}(\rd)\ud\\ \hspace{5.3cm}=\kappa^2\dive \Kd+2\nu\dive\Sd,
\end{cases}
\end{multline}
with initial data
\begin{align}\label{eq:data}
\begin{split}
\rd(0,x)&=\rd^0(x),\\
\rd\ud(0,x)&=\rd^0(x)\ud^0(x).
\end{split}
\end{align}
We refer to \cite{AS} for the motivation of the regularizing terms, see also \cite{LX} where a similar regularizing approximations have been introduced for the compressible Navier-Stokes equations with density dependent viscosity. The approximating viscosity term is defined as 
\[
\Sd=\hd(\rd)\D\ud+g_{\delta}(\rd)\diver\ud\mathbf{I},
\]
with
\begin{equation}\label{eq:hd}
\hd=\rd+\delta(\rd)^{\frac{7}{8}}+\delta(\rd)^{\gamma}, \hspace{1cm}
g_{\delta}=\rd h_{\delta}'(\rd)-h_{\delta}(\rd).
\end{equation}
The approximating dispersive term reads
\[
\dive\Kd=2\rd\nabla\left(\frac{h_{\delta}'(\rd)\dive(h_{\delta}'(\rd)\nabla\sqrt{\rd})}{\sqrt{\rd}}\right).
\]
The coefficient $\Tilde{p}_{\delta}(\rd)$ in the damping term is defined as
\begin{equation*}
    \tilde{p}_\delta(\rd)=e^{-\frac{1}{\delta^4}}\left(\rd^{\frac{1}{\delta^2}}+\rd^{-\frac{1}{\delta^2}}\right).
\end{equation*}
The cold pressure is defined by $P_{\delta}(\rd)=\left(\nu-\sqrt{\nu^2-\kappa^2}\right)\tilde{p}_{\delta}(\rd)\frac{\hd'(\rd)}{\rd}$.
The energy functional for the approximating system \eqref{eq:app} reads
\begin{equation}\label{eq:efunctionalapp}
E_{\delta}(t)=\int_{\T^d}\frac{1}{2}\rd|\ud|^2+\frac{\kappa^2}{2}|\hd'(\rd)\nabla\sqrt{\rd}|^2+\frac{1}{(\gamma-1)}(\rd)^{\gamma}+\fd(\rd)\dd x,
\end{equation}
where $f_{\delta}(\rd)=\rd P_\delta'(\rd)-P_{\delta}(\rd)$ that we observe to be non-negative and strictly convex, see Section 2 in \cite{AS}. In virtue of Theorem 6 in \cite{AS}, there exists a global smooth solution to the Cauchy problem \eqref{eq:app} for smooth initial data $(\rd^0,\ud^0)$. By means of direct computations one infers the energy equality for the system \eqref{eq:app}.
\begin{lemma}[\cite{AS}]\label{lem:EIAPP}
Let $(\rd,\ud)$ be a global smooth solution of \eqref{eq:app}. Then for any $0\leq s<t\leq T$ one has
\begin{equation}\label{ineq:energyAPP}
E_{\delta}(t)+2\nu \int_s^t\int_{\T^d}\hd(\rd)|\D\ud|^2+\gd(\rd)|\dive \ud|^2+\tilde{p}_{\delta}(\rd)|\ud|^2\dd x\dd t'= E_{\delta}(s).
\end{equation}
\end{lemma}
In order to show that smooth solutions to \eqref{eq:app} satisfy a Bresch-Desjardins entropy estimate, we  introduce the effective velocity $\vd=\ud+c\nabla\phi_{\delta}(\rd)$ with $\phi$ defined through $\rd\phi_\delta'(\rd)=\hd'(\rd)$ and for a suitable constant $c$ to be chosen below. Then $(\rd,\vd)$ is a smooth solution of the viscous Euler system
\begin{equation}\label{eq:BD}
\begin{cases}
&\partial_t\rd+\dive(\rd\vd)=c\Delta\hd(\rd),\\
&\partial_t(\rd\vd)+\dive(\rd\vd\otimes\vd)+\nabla(\rd)^{\gamma}+\tilde{\lambda}\nabla p_{\delta}(\rd)-c\Delta(\hd(\rd)\vd)+\tilde{p}(\rd)\vd\\
&-2(\nu-c)\dive(\hd(\rd)\D\vd)-2(\nu-c)\nabla(\gd(\rd)\dive \vd)-\tilde{\kappa}^2\dive \Kd=0,
\end{cases}
\end{equation}
where 
\begin{equation*}
\begin{cases}
\mu=\nu-\sqrt{\nu^2-\kappa^2},\\
\tilde{\kappa}^2=\kappa^2-2\nu c+c^2,\\
\tilde{\lambda}=(\mu-c)/\mu.
\end{cases}
\end{equation*}
The Bresch-Desjardins entropy is then defined as energy functional associated to \eqref{eq:BD},
\begin{equation}\label{eq:BDE}
B_{\delta}(t)=\int_{\T^d}\frac{1}{2}\rd|\vd|^2+\frac{(\rd)^{\gamma}}{(\gamma-1)}+{\tilde{\lambda}}\fd(\rd)+2\tilde{\kappa}^2\left|\hd'(\rd)\nabla\sqrt{\rd}\right|^2\dd x.
\end{equation}
The BD entropy equality arises as energy equality of \eqref{eq:BD}.
\begin{lemma}[\cite{AS}]\label{lem:BDAPP}
Let $(\rd,\ud)$ be a global smooth solution of \eqref{eq:app}. Given $c\in (0,\mu)$, the pair $(\rd,\vd)$ is a global smooth solution of \eqref{eq:BD} and the BD entropy equality is satisfied ,
\begin{align}\label{ineq:BDEapp}
\begin{split}
B_{\delta}(t)&+c \int_0^t\int_{\T^d}\hd(\rd)|\A\vd|^2\dd x\dd s\\
&+2(\nu-c) \int_0^t\int_{\T^d}\hd(\rd)|\D\vd|^2+\gd(\rd)|\dive\vd|^2\dd x\dd s\\
&+ \int_s^t\int_{\T^d}\tilde{p_{\delta}}|\vd|^2+{c\gamma} \hd'(\rd)|\nabla\rd|^2(\rd)^{\gamma-2}+c\tilde{\lambda}\hd'(\rd)|\nabla\rd|^2\fd''(\rd)\dd x\dd s  \\
&+c\tilde{\kappa}^2\int_0^t\int_{\T^d}\hd(\rd)|\nabla^2\phid(\rd)|^2\dd x\dd s+c\tilde{\kappa}^2\int_0^t\int_{\T^d}\gd(\rd)|\Delta\phid(\rd)|^2\dd x\dd s\\
&= B_{\delta}(0).
\end{split}
\end{align}
\end{lemma}
Next, we address existence of smooth solutions to \eqref{eq:app} for a sequence of initial data given by smooth approximations of $(\rho^0,u^0)$ as specified in Theorem \ref{thm:A1}. More precisely, we consider $(\rd^0,\ud^0)$ such that
\begin{align}\label{eq:dataEPS}
\begin{split}
&\rd^0\rightarrow \rho^0 \quad \text{strongly in} \quad L^1(\T^d), \qquad
\{\rd^0\}_\delta \quad \text{uniformly bounded in} \quad L^1\cap L^{\gamma}(\T^d),\\
&\{\hd(\rd^0)\nabla\sqrt{\rd^0}\}_\delta \quad \text{uniformly bounded in} \quad L^2\cap L^{2+\eta}(\T^d),\\
&\hd(\rd^0)\nabla\sqrt{\rd^0}\rightarrow \nabla\sqrt{\rho^0} \quad \text{strongly in }\quad L^2(\T^d),\\
&\{\sqrt{\rd^0}\ud^0\}_{\delta} \quad \text{uniformly bounded in} \quad L^2\cap L^{2+\eta}(\T^d),\\
&\rd^0\ud^0\rightarrow \rho^0 u^0 \quad \text{in} \quad L^1(\T^d),\\
&\fd(\rd^0)\rightarrow 0 \quad \text{strongly in} \quad L^1(\T^d).
\end{split}
\end{align}
Again, Theorem 6 in \cite{AS} guarantees the existence of a sequence of global smooth solutions to the Cauchy problem \eqref{eq:app} with initial data as specified by \eqref{eq:dataEPS}.
\begin{proposition}[\cite{AS}]\label{prop:existencesmooth}
Let $d=2,3$ and $\nu,\kappa >0$, $\gamma>1$ as in Theorem \ref{thm:A1}. Let $(\rd,\ud)$ be a sequence of global smooth solutions to \eqref{eq:app} with initial data \eqref{eq:dataEPS}, then $(\rd,\ud)$ converges to a global weak solution $(\rho,u)$ to \eqref{eq:QNS1} with initial data $(\rho^0,u^0)$. In particular, for any $T>0$ one has
\begin{equation*}
    \sqrt{\rd}\rightarrow \sqrt{\rho}\quad \text{in} \quad L^2(0,T;H^1(\T^d)), \qquad \sqrt{\rd}\ud\rightarrow \sqrt{\rho}u\quad \text{in} \quad L^2(0,T;L^2(\T^d)).
\end{equation*}
\end{proposition}
We are now in position to prove Theorem \ref{thm:A1}. 
\begin{proof}[Proof of Theorem \ref{thm:A1}]
Given initial data $(\rho^0,u^0)$, we construct a sequence of smooth initial data $(\rd^0,\ud^0)$ satisfying \eqref{eq:dataEPS} by mollification. Theorem 6 in \cite{AS} states the existence of a sequence of smooth solutions  $(\rd,\ud)$ to \eqref{eq:app}. Proposition \ref{prop:existencesmooth} yields convergence towards a weak solution $(\rho,u)$ of \eqref{eq:QNS1}. It remains to pass to the limit in \eqref{eq:BD} and \eqref{ineq:BDEapp} to show \eqref{ineq:energy} and \eqref{ineq:BDE} respectively. First, we show \eqref{ineq:energy}. To that end,
we notice that 
\begin{equation*}
    \hd(\rd)|\D\ud|^2+g_{\delta}(\rd)|\diver \ud|^2\geq\left(\rd+\delta\left(\frac{5}{8}\rd^{\frac{7}{8}}+\rd^{\gamma}\right)\right)|\D \ud|^2.
\end{equation*}
For all $0\leq t\leq T$, one has
\begin{equation*}
    \int_0^t\int_{\T^d}\rd\left|\D \ud\right|^2\dd x \dd s\leq \int_0^t\int_{\T^d}\hd(\rd)\left|\D\ud\right|^2+\gd(\rd)|\diver \ud|^2\dd x \dd s.
\end{equation*}
In virtue of \eqref{ineq:energyAPP}, we conclude that $\sqrt{\rd}\D\ud\in L^2(0,T;L^2(\T^d))$ uniformly bounded. Hence, there exists $\Sb\in L^2(0,T;L^2(\T^d))$ such that $\sqrt{\rd}\D\ud\rightharpoonup \Sb$ in $L^2(0,T;L^2(\T^d))$, up to passing to subsequences. By lower semi-continuity of norms, we obtain that 
\begin{equation}\label{eq:energyAUX1}
     \int_0^t\int_{\T^d}|\Sb|^2\dd x\dd s\leq \liminf_{\delta\rightarrow 0} \int_0^t\int_{\T^d}\hd(\rd)\left|\D\ud\right|^2+\gd(\rd)|\diver \ud|^2\dd x \dd s.
\end{equation}
Similarly,
\begin{equation}\label{eq:energyAUX2}
    \begin{aligned}
    E_{\T^d}(\rho(t),u(t))&\leq \liminf_{\delta\rightarrow 0}\int_{\T^d}\frac{1}{2}\rd|\ud|^2+\frac{\kappa^2}{2}|\nabla\sqrt{\rd}|^2+\frac{1}{(\gamma-1)}\rd^{\gamma}\dd x\\
    &\leq \liminf_{\delta\rightarrow 0}E_{\delta}(\rd(t),\ud(t)),
    \end{aligned}
\end{equation}
where we exploit that $|\nabla\sqrt{\rd}|\leq |\hd'(\rd)\nabla\sqrt{\rd}|$ and that $f(\rd)$ is non-negative for $\delta$ sufficiently small. Finally, summing up \eqref{eq:energyAUX1}, \eqref{eq:energyAUX2} and applying \eqref{ineq:energyAPP} yields
\begin{equation*}
    E_{\T^d}(\rho(t),u(t))+\int_{0}^t\int_{\T^d}|\Sb|^2\dd x \dd s
    \leq \liminf_{\delta\rightarrow 0}E_{\delta}(\rd^0,\ud^0)=E_{\T^d}(\rho^0,u^0),
\end{equation*}
where we used that $\tilde{p}_\delta(\rd)$ is non-negative for the first inequality and \eqref{eq:dataEPS} in the last identity. Next, we check that $\mathbf{S}$ satisfies \eqref{eq:visc_symm}. As $\sqrt{\rd}$ converges strongly to $\sqrt\rho$ in $L^2(0,T;L^2(\T^d))$ in virtue of Proposition \ref{prop:existencesmooth} and $\sqrt{\rd}\D\ud$ converges weakly to $\mathbf{S}$ in $L^{\infty}(0,T;L^2(\T^d))$, one has that $\rd\D\ud$ converges to $\sqrt{\rho}\Sb$ in $\mathcal{D}'([0,T)\times\T^d)$. Let $\phi \in  \mathcal{D}((0,T)\times \T^d)$ and consider
\begin{equation*}
\langle \sqrt{\rd}\D\ud, \phi \rangle=\left\langle(\D(\rd\ud)), \phi \right\rangle- 2\left\langle\left(\sqrt{\rd}\ud\otimes\nabla\sqrt{\rd}\right)^{sym}, \phi \right\rangle.
 \end{equation*}
As $\sqrt{\rd}\rightarrow \sqrt\rho$ and $\nabla\sqrt{\rd}\rightarrow \nabla\sqrt{\rho}$ both strongly in $L^{2}(0,T;L^2(\T^d))$ and $\sqrt{\rd}\ud\rightharpoonup \sqrt{\rho}u$ weakly-$\ast$ in $L^{\infty}(0,T;L^2(\T^d))$ we conclude that $\Sb$ satisfies \eqref{eq:visc_symm}. Second, we show \eqref{ineq:BDE} by proceeding similarly. As $\hd(\rd)\geq \rd$, we infer from \eqref{ineq:BDEapp} that $\sqrt{\rd}\A\vd\in L^2(0,T;L^2(\T^d))$ uniformly bounded. We notice that $\A\vd=\A\ud$. Hence, there exists $\Ab$ such that  $\sqrt{\rd}\A\ud\rightharpoonup \Ab$ up to subsequences and
\begin{equation*}
    \int_0^t\int_{\T^d}|\Ab|^2\dd x \dd s\leq \liminf_{\delta\rightarrow 0}\int_0^t\int_{\T^d}\hd(\rd)|\A\vd|^2\dd x \dd s.
\end{equation*}

In particular, proceeding as for the identification of $\mathbf{S}$, we obtain that $\mathbf{T}=\mathbf{S}+\mathbf{A}$ satisfies \eqref{eq:visc_tens}. 
Further, we observe that 
\begin{equation*}
    \hd(\rd)|\nabla^2\phid(\rd)|^2+\gd(\rd)|\Delta\phid(\rd)|^2\geq \frac{5}{8}\hd(\rd)|\nabla^2\phid(\rd)|^2.
\end{equation*}
Thus, from Lemma 4 in \cite{AS} we obtain that there exists $C>0$ such that 
\begin{equation*}
    \int_0^t\int_{\T^d}\left|\nabla^2\sqrt{\rd}\right|^2\dd x\dd s\leq C\int_0^T\int_{\T^d}\hd(\rd)\left|\nabla^2\phid(\rd)\right|^2+\gd(\rd)|\dive \ud|^2\dd x\dd s.
\end{equation*}
It follows that, up to passing to subsequences, $\nabla^2\sqrt{\rd}$ converges weakly to $\nabla^2\sqrt{\rho}$ in $L^2(0,T;L^2(\T^d))$. Moreover, \eqref{ineq:BDEapp} together with 
\begin{equation*}
    \int_0^t\int_{\T^d}|\nabla\rd|^2\rd^{\gamma-2}\dd x\dd s \leq \int_0^t\int_{\T^d}\left|\hd'(\rd)\nabla{\rd}\right|^2(\rd)^{\gamma-2}\dd x\dd s
\end{equation*}
implies that $\nabla\rd^{\frac{\gamma}{2}}\in L^2(0,T;L^2(\T^d))$ uniformly bounded, hence up to passing to subsequences
\begin{equation*}
    \nabla\rd^{\frac{\gamma}{2}}\rightharpoonup \nabla\rho^{\frac{\gamma}{2}} \qquad \text{in} \quad L^2(0,T;L^2(\T^d)).
\end{equation*}
Summing up and exploiting lower semi-continuity of norms we have
\begin{equation}\label{eq:dissipationBD}
\begin{aligned}
    &\int_0^t\int_{\T^d}c|\Ab|^2+c\Tilde{\kappa}^2\left|\nabla^2\sqrt{\rho}\right|^2+\frac{c}{\gamma}\left|\nabla\rho^{\frac{\gamma}{2}}\right|^2\dd x \dd s\\
    &\leq \liminf_{\delta\rightarrow 0}\Big( \int_0^t\int_{\T^d}c\hd(\rd)|\A\vd|^2+c\gamma\hd'(\rd)\rd^{\gamma-2}|\nabla\rd|^2\dd x\dd s\\
    &+C\int_{0}^t\int_{\T^d}c\Tilde{\kappa}^2 \hd(\rd)|\nabla^2\phid(\rd)|^2+\gd(\rd)|\Delta\rd|^2 \dd x \dd s \Big)
    \end{aligned}
\end{equation}
By exploiting lower-semicontinuity of norms,  \eqref{ineq:BDEapp}, \eqref{eq:dissipationBD}, as well as $B_{\delta}(0)\rightarrow B(0)$ due to \eqref{eq:dataEPS} we infer \eqref{ineq:BDE}. For that purpose, we notice that $\Tilde{p}_\delta(\rd)$ and $f_{\delta}''(\rd)$ are non-negative for $\delta$ sufficiently small
by construction.
\end{proof}


\section{Strichartz estimates for the acoustic wave system}\label{app:Strichartz}

The main purpose of this appendix is to give a proof of Proposition \ref{prop:Strichartz} (see Proposition \ref{prop:Strichartz final} below), that is we want to study the dispersive properties satisfied by solutions to system \eqref{eq:symAW}. Even if the paper only studies the three dimensional setting for the sake of completeness the whole analysis is carried out in the general $d$-dimensional setting ($d\geq 2)$.  As already mentioned, for $\eps=1$, the dispersive analysis associated to the operator $H_1$ has been carried out in \cite{GNT05, GNT07, GNT09}. In this paper, we need to carefully track down the $\eps$-dependence on the estimates as the (scaled) Mach number $\eps$ not only determines a time scale but also a frequency threshold such that the operator behaves differently. This is due to the non-homogeneity of the dispersion relation and is opposite to the analysis of low Mach number limit in classical fluid dynamics where the Mach number $\eps$ only determines the time scale. The dispersive analysis for non-homogeneous symbols has been investigated in more general framework also in \cite{GPW08,COX, GHN17}. Homogeneous Besov spaces suit best for the purpose of proving the refined Strichartz estimates while simultaneously tracking the $\eps$-dependence. \\
We structure the exposition of this Appendix as follows. In Proposition \ref{prop:GNT} we recall the stationary phase estimate proved in \cite{GNT05}. The scaled dispersive estimate associated to the operator $H_\eps$ in \eqref{op:Heps} is then given in Corollary \ref{coro: dispersive estimate 2}. By using the dispersive estimate it is then possible to derive the frequency localized Strichartz estimates of Lemma \ref{lemma:b10}. By summing over all frequencies we then obtain Proposition \ref{prop:Strichartz_A}.
For the convenience of the reader, the final estimates used in the body of this paper, namely Section \ref{sec:acw}, are stated in terms of Sobolev spaces, see Proposition \ref{prop:Strichartz final}. 

\subsection{Besov spaces and embeddings}
We recall the embeddings relating homogeneous Besov spaces to Lebesgue space $L^q$ and homogeneous Sobolev spaces $\dot{W}^{s,q}$.
\begin{lemma}\label{lem:embedding}
Let $d\geq 1$.
\begin{enumerate}[(i)]
\item For any $s\in \R$, we have $\dot{B}_{2,2}^{s}(\R^d)=\dot{H}^{s}(\R^d)$.
    \item For $q\in [2,\infty)$, the homogeneous Besov space $\dot{B}_{q,2}^{0}(\R^d)$ is continuously embedded in $L^q(\R^d)$ and $L^{q'}(\R^d)$ embeds continuously in $\dot{B}_{q',2}^{0}(\R^d)$.
    \item For $s\in \R$ and $q\in[2,\infty)$, one has
    \begin{equation*}
    \begin{aligned}
        \dot{B}_{q,2}^s(\R^d)\hookrightarrow \dot{W}^{s,q}(\R^d) \hookrightarrow  \dot{B}_{q,q}^s(\R^d),\\
        \dot{B}_{q',q'}^s(\R^d)\hookrightarrow \dot{W}^{s,q'}(\R^d) \hookrightarrow  \dot{B}_{q',2}^s(\R^d).
        \end{aligned}
    \end{equation*}
\end{enumerate}
\end{lemma}

The proof of statement (iii) can be found e.g. in \cite{BL}, see Theorem 6.4.4. The first and second statements are direct consequences of the third. 

We state a useful property for space-time Besov spaces.
\begin{lemma}\label{lem:spacetimeBesov}
Let $d\geq 1$, $T>0$ and $p,q,r\in[1,\infty)$. Then, for all $p\geq r$ one has
\begin{equation*}
    \|u\|_{L^p(0,T;\dot{B}_{q,r}^s(\R^d))}\leq \left\|2^{ks}\|P_{2^k}u\|_{L^p(0,T;L^q(\R^d))}\right\|_{l^r(\Z)},
\end{equation*}
while for $p\leq r$,
\begin{equation*}
     \left\|2^{ks}\|P_{2^k}u\|_{L^p(0,T;L^q(\R^d))}\right\|_{l^r(\Z)}\leq \|u\|_{L^p(0,T;\dot{B}_{q,r}^s(\R^d))}.
\end{equation*}
\end{lemma}
The statements follow upon applying Minkowski's integral inequality.

\subsection{Dispersive estimate}\label{subch:disp}
In what follows we are going to prove the  $L^{\infty}-L^1$ dispersive estimate associated for the semigroup $e^{itH_{\eps}}$. For the convenience of the reader, we recall the stationary phase estimate in \cite{GNT05}. Here and below, we adopt the notation $r=|\xi|$.
\begin{proposition}[\cite{GNT05}]\label{prop:GNT}
Let $\phi(r)\in C^{\infty}(0,\infty)$ satisfy the following.
\begin{enumerate}[(i)]
\item $\phi'(r),\phi''(r)>0 $ for all $r>0$,
\item $\phi'(r)\sim \phi'(s)$ and $\phi''(r)\sim \phi''(s)$ for all $0<s<r<2s$,
\item $\left|\phi^{(k+1)}(r)\right|\lesssim \frac{\phi'(r)}{r^k}$ for all $r>0$ and $k\in \N$.
\end{enumerate}
Let $\chi(r)$ be a dyadic cut-off function with support around $r \sim R$ and that satisfies
\[
|\chi^{(k+1)}(r)|\lesssim R^{-k}.
\]
These estimates are supposed to hold uniformly for $r$ and $R$, but may depend on $k$. Then if
\begin{equation*}
I_\phi(t, x, R):=\int_{\R^d}e^{it\phi(|\xi|)+ix\cdot\xi}\chi(|\xi|)\,d\xi,
\end{equation*}
we have
\begin{equation}\label{est:GNT}
\sup_{x\in\R^d}\left|I_{\phi}(t,x,R)\right|\lesssim t^{-\frac{d}{2}}\left(\frac{\phi'(R)}{R}\right)^{-\frac{d-1}{2}}\left(\phi''(R)\right)^{-\frac{1}{2}}.
\end{equation}
\end{proposition}

Several observations are in order. We define
\begin{equation}\label{eq:h}
    h(r)=\det\left(\Hess(\phi( r))\right),
\end{equation} 
and exploiting that $\phi$ is a radial function we compute
\begin{equation*}
    h(r)= \left(\frac{\phi'(r)}{r}\right)^{d-1}\phi''(r),
\end{equation*}
so that the right hand side of \eqref{est:GNT} involves $h(R)^{-1/2}$. This is consistent with the general theory for stationary phase estimates, see for instance Theorem 7.7.6 in \cite{H90}.
Furthermore, from Proposition 2 in \cite{COX} it follows that the dispersive estimate \eqref{est:GNT} is sharp in the sense that there exists $t_0$ and $R_0$ such that for all $|t|>|t_0|$ and $R>R_0$ there exists $C>0$ such that 
\begin{equation}\label{est:sharp}
\sup_{x\in\R^d}\left|I_{\phi}(t,x,R)\right|\geq C t^{-\frac{d}{2}}\left(\frac{\phi'(R)}{R}\right)^{-\frac{d-1}{2}}\left(\phi''(R)\right)^{-\frac{1}{2}}.
\end{equation}
In \cite{GNT05}, the estimate \eqref{est:GNT} has been applied to the pseudo-differential operator $H_1$, i.e. $\phi(r)=r \sqrt{1+\kappa^2r^2}$. We remark that the dispersive estimate for the symbol $\omega$ defined by the Bogliubov dispersion relation \eqref{eq:Bog_sp} can be obtained from estimate \eqref{est:GNT} by defining
\begin{equation*}
\phi_{\eps}(r)=\frac{1}{\eps^2}\phi\left(\eps r\right).
\end{equation*}
Indeed, we notice that  $\omega(r)=\phi_{\eps}(r)$ and after a computation that 
\begin{equation}\label{eq:Hess}
    \det\left(\Hess(\phi_{\eps}(r))\right)=h(\eps r),
\end{equation}
as well as 
\begin{equation}\label{eq:oscINT}
I_{\phi_{\eps}}(t, x, R):=\int_{\R^d}e^{ix\xi+it{\phi_{\eps}}(|\xi|)}\chi(|\xi|)d\xi=\eps^{-d}I_\phi(\frac{t}{\eps^2}, \frac{x}{\eps}, \eps R),
\end{equation}
by rescaling. We remark that the scaling affects the support of frequencies. Finally, to track down the $\eps$-dependence in the dispersive estimate it is enough to study the properties of the Hessian matrix of $\phi_{\eps}$ in terms of its determinant $h(\eps r)$.  

\begin{lemma}\label{lem:B3}
Let $h$ be defined as in \eqref{eq:h}.
There exists $C>0$ such that for any $\lambda\in [0,\infty]$,
\begin{equation}\label{eq:boundh}
    \frac{1}{C\kappa}r^{\frac{d-2}{2}}\leq h(\lambda)^{-\frac{1}{2}}\leq C \frac{1}{\kappa}\left( \frac{ \lambda}{\sqrt{1+(\kappa \lambda)^2}}\right)^{\frac{d-2}{2}}.
\end{equation}
For $d=2$, there exists $C>0$ such that for any $\lambda\in [0,\infty)$,
    \begin{equation*}
        \frac{1}{\sqrt{3}\kappa}\leq h(\lambda)^{-\frac{1}{2}}\leq \frac{C}{\kappa}.
    \end{equation*}
\end{lemma}
We shall use these estimates for $\lambda=\eps|\xi|$. We notice that the estimates blow up as $\kappa$ goes to $0$. Here, we consider $\kappa>0$ fixed, see also Remark \ref{rem:kappa}.
\begin{proof}
This follows from immediate computations.
\end{proof}

The information on the (scaled) function $h$ allows to derive the dispersive estimate for the symbol $\phi_{\eps}$.

\begin{corollary}\label{coro: dispersive estimate 1}
Let $\phi_{\eps}(r)=\frac{1}{\eps}r\sqrt{1+(\eps\kappa)^2r^2}$, $R>0$ be given and let $\chi(r)\in C_c(0,\infty)$ be as in Proposition \ref{prop:GNT}. Then there exists a constant $C>0$ such that
\begin{equation}\label{eq:disp1}
\sup_{x\in \R^d}\left| I_{\phi_{\eps}}(t,x,R)\right|\leq C t^{-\frac{d}{2}}\left(\frac{\phi_{\eps}'(R)}{R}\right)^{-\frac{(d-1)}{2}}\phi_{\eps}''(R)^{-\frac{1}{2}}.
\end{equation}
In particular, this implies there exists $C>0$ independent from $\eps$ such that,
\begin{equation}\label{eq:disp2}
\sup_{x\in \R^d}\left| I_{\phi_{\eps}}(t,x,R)\right|\leq C t^{-\frac{d}{2}}h(\eps R)^{-\frac12}\leq C' t^{-\frac{d}{2}}.
\end{equation}
\end{corollary}

\begin{proof}
For fixed $\eps>0$, we may exploit the dispersive estimate of Proposition \ref{prop:GNT} upon using \eqref{eq:oscINT}. More precisely, we obtain 
\begin{equation*}
\begin{aligned}
\sup_{x\in \R^d} \left| I_{\phi_\eps}(t, x, R)\right|&=\eps^{-d}\sup_{x\in \R^d} \left| I_\phi(\frac{t}{\eps^2}, \frac{x}{\eps}, \eps R)\right|\\
& \leq C t^{-\frac{d}{2}}\left(\frac{\phi'(\eps R)}{\eps R}\right)^{-\frac{d-1}{2}}\left(\phi''(\eps R)\right)^{-\frac{1}{2}}= C t^{-\frac{d}{2}}h(\eps R)^{-\frac{1}{2}}.
\end{aligned}
\end{equation*}
Estimate \eqref{eq:disp1} then follows by applying \eqref{eq:Hess}.
To conclude the second estimate, it is enough to observe that for $d\geq 2$ there exists $C>0$ such that $h(r)^{-\frac{1}{2}}\leq C$ uniformly on $(0,\infty)$ as consequence of Lemma \ref{lem:B3}.
\end{proof}

The estimate in \eqref{eq:disp2} implies that the operator $H_\eps$ has the same dispersive properties as the Schr\"odinger operator. As a consequence \eqref{eq:disp2} would yield Schr\"odinger type dispersive estimates for frequency localized functions. However, from \eqref{eq:boundh} we shall infer that in fact, for $d>2$, we can derive better estimates, due to the regularizing effect of $\frac{(\eps r)}{\sqrt{1+(\eps r)^2}}$ when $\eps r$ is small. This has already been pointed out in \cite{GNT05}  for the operator $e^{itH_1}$ and is explained by a different curvature of the geometric surface $|\xi|\sqrt{1+|\xi|^2}$ with respect to $|\xi|^2$.
We reformulate this observation in the next Corollary.
\begin{corollary}\label{coro: dispersive estimate 2}
Let $d\geq 2$, $\phi_{\eps}(r)=\frac{1}{\eps}r\sqrt{1+(\eps\kappa)^2r^2}$, $R>0$ be given and let $\chi(r)\in C_c(0,\infty)$ be as in Proposition \ref{prop:GNT}. Then there exists a constant $C>0$ such that
\begin{equation}\label{eq:disp3}
\sup_{x\in \R^d}\left| I_{\phi_{\eps}}(t, x, R)\right| \leq \frac{C}{\kappa^{\frac{d}{2}}} t^{-\frac{d}{2}}\left(\frac{\eps \kappa R}{\sqrt{1+(\eps\kappa R)^2}}\right)^{\delta},
\end{equation}
for any $0\leq\delta\leq\frac{d-2}{2}$.
\end{corollary}
\begin{remark}\label{rem:kappa}
We emphasize that the RHS in the estimate \eqref{eq:disp3} blows up as $\kappa$ goes to $0$. This reflects the contribution of the quantum pressure term to the dispersion relation. In the absence of the quantum pressure, i.e. $\kappa=0$ one recovers a linear dispersion relation for which wave-type dispersive estimates are to be expected and \eqref{eq:disp3} being of Schr\"odinger type does not hold. However, here we consider the system \eqref{eq:QNS} for fixed and bounded $\kappa$ and thus we set $\kappa=1$ in the following, see also the scaling chosen in Section \ref{subsec:scaling}.
\end{remark}

This motivates to define the pseudo-differential operator $U_{\eps}$ corresponding to the Fourier multiplier
\begin{equation}\label{eq:Fmultiplier}
m(\eps\xi)=\frac{(\eps \xi)}{\sqrt{1+(\eps|\xi|)^2}}, \quad \text{as} \quad U_{\eps}:=(-\eps^2\Delta)^{\frac{1}{2}}(1-\eps^2\Delta)^{-\frac{1}{2}}.
\end{equation}
In particular, this allows for the gain of the factor $\eps^{\delta}$ in the estimate at the expense of a factor $R^{\delta}$ corresponding to a loss of derivatives, see inequality \eqref{eq:delta_der}.

\subsection{Strichartz estimates}\label{subs:Strichartz}
Next, we infer the needed Strichartz estimates from the dispersive estimate \eqref{eq:disp3}. First, we recall the definition of admissible exponents. 
\begin{definition}\label{def:betagamma}
We say the pair of exponents $(p, q)$ is Schr\"odinger admissible if $2\leq p, q\leq\infty$,
\begin{equation*}
\frac2p+\frac{d}{q}=\frac{d}{2}
\end{equation*}
and $(p, q, d)\neq(2, \infty, 2)$.
\end{definition}

The first step consists in showing a pointwise in time estimate. 
\begin{lemma}[Pointwise estimate]\label{lem:pointwise}
For fixed $\eps>0$ and  $R>0$, let $f\in L^{1}(\R^d)$ such that $supp(\hat{f})\subset \{\xi\in\R^d\, : \, \frac{1}{2}R\leq |\xi|\leq 2R \}$. The following estimate holds for any $2\leq q \leq \infty$ and $0\leq \delta\leq\frac{d-2}{2}$:
\begin{equation*}
\left\|\eith f\right\|_{L^q(\R^d)}\leq C t^{-d\left(\frac12-\frac1q\right)}\left\|U_{\eps}^{\delta\left(1-\frac2q\right)}f\right\|_{L^{q'}(\R^d)},
\end{equation*}
and consequently
\begin{equation}\label{eq:delta_der}
\left\|\eith f\right\|_{L^q(\R^d)}\leq C t^{-d\left(\frac12-\frac1q\right)}(\eps R)^{\delta\left(1-\frac2q\right)} \left\|f\right\|_{L^{q'}(\R^d)}.
\end{equation}
\end{lemma}

\begin{proof}
The operator $\eith$ is unitary on $L^2$ therefore 
\begin{equation*}
\|\eith f\|_{L^2(\R^d)}=\|f\|_{L^2(\R^d)},
\end{equation*}
Furthermore, Corollary \ref{coro: dispersive estimate 2} guarantees that there exists $C_2>0$ not depending on $\eps,R$ 
  \begin{equation*}
\|\eith f\|_{L^{\infty}(\R^d)}\leq C_2 t^{-\frac{d}{2}}\|\Ueps^{\delta}f\|_{L^1(\R^d)}.
\end{equation*}
By a standard interpolation argument we conclude the proof. Estimate \eqref{eq:delta_der} follows from
\begin{equation*}
\left(\frac{\eps R}{\sqrt{1+\eps^2R^2}}\right)^{\delta\left(1-\frac2q\right)}\leq(\eps R)^{\delta\left(1-\frac2q\right)}.
\end{equation*}
\end{proof}

Next, we show Strichartz estimates localized in frequencies on dyadic blocks. 

\begin{lemma}\label{lemma:b10}
For $d\geq2$, $\eps, R>0$ and $0<\delta\leq \frac{d-2}{2}$, let $f\in L^{2}(\R^d)$ and $F\in L^{p'}(0,T;L^{q'}(\R^d))$ such that $supp(\hat{f}), supp(\hat F(t))\subset \{\xi\in\R^d\, : \, \frac{1}{2}R\leq |\xi|\leq 2R \}$ Then there exists a constant $C>0$ independent from $T,\eps $ such that for any $(p,q)$, $(p_1,q_1)$ admissible pairs,
\begin{equation}\label{eq:Strichartz-lh}
\|\eith f\|_{L^p(0,T;L^q(\R^d))}\leq C  \|U_{\eps}^{\delta\left(\frac12-\frac1q\right)}f\|_{L^2(\R^d)},
\end{equation}
\begin{equation}\label{eq:Strichartz-lh1}
\left\|\int_{\R}e^{-it\Heps}F(t)\text{d}t\right\|_{L^2(\R^d)}\leq C \|U_{\eps}^{\delta\left(\frac12-\frac1q\right)}F\|_{L^{p'}(0,T;L^{q'}(\R^d))}.
\end{equation}
Moreover,
\begin{equation}\label{eq:Strichartz-lnh}
\begin{aligned}
\left\|\int_{\R}e^{i(t-s)\Heps}F(s)\dd s\right\|_{L^p(0,T;L^{q}(\R^d))}\leq C^2  \left\| U_{\eps}^{\delta\left(1-\frac1q-\frac{1}{q_1}\right)}F\right\|_{L^{p_1'}(0,T;L^{q_1'}(\R^d))},\\
\left\|\int_{s<t}e^{i(t-s)\Heps}F(s)\dd s\right\|_{L^p(0,T;L^q(\R^d))}\leq C^2  \left\| U_{\eps}^{\delta\left(1-\frac1q-\frac{1}{q_1}\right)}F\right\|_{L^{p_1'}(0,T;L^{q_1'}(\R^d))}.
\end{aligned}
\end{equation}
\end{lemma}

\begin{proof}
Given \eqref{eq:disp3} and considering the fact that $\eith$ is an isometry on $L^2(\R^d)$, we observe that Theorem  1 of \cite{KT98} applies. We notice that the constants in the estimates \eqref{eq:Strichartz-lh} are identical as coming from an abstract duality argument. 
\end{proof}

We remark that for $\eps=1$, we recover the Strichartz estimates provided by Theorem 2.1 in \cite{GNT05}. 

\begin{proposition}\label{prop:Strichartz_A}
Let $d\geq 2$, $\eps>0$,  $0\leq\delta\leq \frac{d-2}{2}$.  Then there exists a constant $C>0$ independent from $T,\eps $ such that for any $(p,q)$, $(p_1,q_1)$ admissible pairs,
\begin{equation}\label{eq:Strichartz-h}
\|\eith f\|_{L^p(0,T;\dot{B}_{q,2}^{0}(\R^d))}\leq C \|U_{\eps}^{\delta\left(\frac12-\frac1q\right)}f\|_{L^2(\R^d)},
\end{equation}
and
\begin{equation}\label{eq:Strichartz-nh}
\left\|\int_{s<t} e^{i(t-s)\Heps} F(s)\dd s\right\|_{L^p(0,T;\dot{B}_{q,2}^{0}(\R^d))}\leq   C  \| U_{\eps}^{\delta\left(1-\frac1q-\frac{1}{q_1}\right)}f \|_{L^{p_1'}(0,T;\dot{B}_{q_1',2}^0(\R^d))}.
\end{equation}
\end{proposition}

\begin{proof}
Fix $R>0$, by scaling $t'=\frac{t}{R^2}$ and $x'=\frac{x}{R}$, we achieve that the projection $P_R(e^{it'\Heps}f)(t',x')$ is spectrally supported in the annulus $\{\xi\in\R^d\, : \, \frac{1}{2}\leq |\xi|\leq 2 \}$. In the following we use the subscript $L_{\mathbf{\cdot}}^p$ for Lebesgue spaces to indicate the variable w.r.t which the norm is computed. We infer from \eqref{eq:Strichartz-lh} that 
\begin{equation*}
\begin{aligned}
\left\|P_R\left(e^{it\Heps}f\right)\right\|_{L_t^pL_x^q}=R^{\frac{d}{q}+\frac{2}{p}}\left\|P_1\left(e^{it'\Heps}f\right)\right\|_{L_{t'}^pL_{x'}^q}&\\
    \leq C R^{\frac{d}{q}+\frac{2}{p}} \left\|P_1\left(U_\eps^{\frac{\beta(q)}{2}}f\right)\right\|_{L_{x'}^2}
    \leq C \left\|P_R\left(U_\eps^{\frac{\beta(q)}{2}}f\right)\right\|_{L_x^2}
    \end{aligned}
\end{equation*}
for any admissible pair $(p,q)$, namely such that  $\frac{2}{p}+\frac{d}{q}=\frac{d}{2}$. Similarly, the bound \eqref{eq:Strichartz-lnh} implies that for admissible pairs $(p,q)$ and $(p_1,q_1)$ , we have
\begin{equation*}
    \left\|P_R\left(\int_{s<t}e^{i(t-s)\Heps}F(s)\dd s\right)\right\|_{L_t^pL_x^q}\leq C^2  \left\| P_R\left(U_{\eps}^{\delta\left(1-\frac1q-\frac{1}{q_1}\right)}F\right)\right\|_{L_t^{p_1'}L_x^{q_1'}}.
\end{equation*}
Hence, given an admissible pair $(p,q)$ and setting $R=2^k$, we compute
\begin{align*}
\left\|\eith f\right\|_{L^p\dot{B}_{q,2}^s}&\leq \left\|2^{ks}\|P_{2^k}\left(\eith f\right) \|_{L_t^pL_x^q}\right\|_{l_k^2}\\
& \leq C\left\|2^{ks}\left\|P_{2^k}\left(U_\eps^{\delta(\frac12-\frac1q)}f\right) \right\|_{L_x^2}\right\|_{l_k^2}\\
& \leq C \left\|\Ueps^{\delta(\frac12-\frac1q)}f\right\|_{\dot{B}_{2,2}^s},
\end{align*}
where we have used the inequalities of Lemma \ref{lem:spacetimeBesov} in the first and third inequality respectively and \eqref{eq:Strichartz-lh} in the second.
Similarly, we proceed for \eqref{eq:Strichartz-lh1}. Indeed,
\begin{align*}
\left\|\int_{\R}e^{i(t-s)\Heps} F(s)\dd s\right\|_{L^p\dot{B}_{q,2}^s}&\leq  \left\|2^{ks}\left\|P_{2^k}\left(\int_{\R}e^{i(t-s)\Heps} F(s)\dd s\right)\right\|_{L_t^pL_x^q}\right\|_{l_k^2}\\
& \leq C\left\|2^{ks}\left\|P_{2^k}\left(U_\eps^{\delta\left(1-\frac1q-\frac{1}{q_1}\right)} F\right)\right \|_{L_t^{p_1'}L_x^{q_1'}}\right\|_{l_k^2}\\
& \leq C \left\|\Ueps^{\delta\left(1-\frac1q-\frac{1}{q_1}\right)}F\right\|_{L^{p_1'}\dot{B}_{q_1',2}^s}.
\end{align*}
\end{proof}

The final estimates follow upon observing that the presence of the operator $\Ueps$ may be exploited to gain a factor $\eps$ as shown in \eqref{eq:delta_der}. For the purpose of Section \ref{sec:acw}, it is sufficient to state the estimates in non-homogeneous Sobolev spaces rather than in the more general framework of Besov spaces. 
\begin{proposition}\label{prop:Strichartz final}
Let $d\geq 2$, fix $\eps>0$ and $s\in \R$. There exists a constant $C>0$ independent from $T,\eps $ such that for any $(p,q)$ admissible pair and any $\alpha_0\in [0,\frac{d-2}{2}(\frac12-\frac1q)]$, the following hold true,
\begin{equation}\label{eq:Strichartz-f1}
\|\eith f\|_{L^p(0,T;W^{-s-\alpha_0,q}(\R^d))}\leq C \eps^{\alpha_0}\|f\|_{H^{-s}(\R^3)}.
\end{equation}
Moreover for any $(p_1,q_1)$ admissible and $\alpha_1\in[0, \frac{d-2}{2}(1-\frac1q-\frac{1}{q_1})]$ it holds
\begin{equation}\label{eq:Strichartz-f2}
\left\|\int_{s<t}e^{i(t-s)H_{\eps}}F(s)\dd s\right\|_{L^p(0,T;W^{-s-\alpha_1,q}(\R^d))}\leq   C \eps^{\alpha_1} \| F \|_{L^{p_1'}(0,T; W^{-s,q_1'}(\R^d))}.
\end{equation}
\end{proposition}
We observe that $\alpha_0\leq \alpha_1$ as $q_1\geq 2$.
\begin{proof}
We notice that $\Ueps$ defined by \eqref{eq:wellprepared} is such that its symbol satisfies
\begin{equation}\label{eq:Uepsalpha}
    m(\eps\xi)\leq \eps\xi.
\end{equation}
It follows that for any $q\in(1,\infty)$, $\alpha\geq 0$ that  \begin{equation*}
   \|\Ueps^{\alpha}f\|_{\dot{B}_{q,2}^s}\leq C \eps^{\alpha}\|f\|_{\dot{B}_{q,2}^{s+\alpha}}. 
\end{equation*}
Next we recall from Lemma \ref{lem:embedding} that $\dot{B}_{q,2}^0(\R^d)\hookrightarrow L^q(\R^d)$ for any $q\in [2,\infty)$ and that $\dot{B}_{2,2}^{-s}(\R^d)=\dot{H}^{-s}(\R^d)$. Hence, \eqref{eq:Strichartz-h} yields
\begin{equation}\label{eq:S1}
\|\eith f\|_{L^p(0,T;L^q(\R^d))}\leq C \eps^{\alpha_0} \|f\|_{\dot{H}^{\alpha_0}(\R^d)}
\end{equation}
for any $\alpha_0\in [0,\frac{d-2}{2}(\frac12-\frac1q)]$. Given $s\geq 0$, applying \eqref{eq:S1} to $\Tilde{f}=(1-\Delta)^{-\frac{s+\alpha_0}{2}}f$ we obtain the estimate
\begin{equation*}
\|\eith f\|_{L^p(0,T;W^{-s-\alpha_0,q}(\R^d))}\leq C \eps^{\alpha_0} \|f\|_{{H}^{-s}(\R^d)}
\end{equation*}
for any admissible pair $(p,q)$ and $\alpha_0\in [0,\frac{d-2}{2}(\frac12-\frac1q)]$. It remains to show \eqref{eq:Strichartz-f2}. Lemma \ref{lem:embedding} provides the embedding $\dot{W}^{s,q'}(\R^d)\hookrightarrow\dot{B}_{q',2}^{s}(\R^d)$ for $s\in \R$ and $q\in [2,\infty)$, we conclude from \eqref{eq:Strichartz-nh} that for any $(p_1,q_1)$ admissible, we have
\begin{equation}\label{eq:S2}
\left\|\int_{s<t} e^{i(t-s)\Heps} F(s)\dd s\right\|_{L^p(0,T;L^q(\R^d))}\leq C {\eps}^{\alpha_1} 
\| F \|_{L^{p_1'}(0,T;\dot{W}^{\alpha_1,q_1'}(\R^d))}
\end{equation}
provided $\alpha_1\in [0,\frac{d-2}{2}(1-\frac1q-\frac{1}{q_1})]$. Applying \eqref{eq:S2} to $\Tilde{F}=(1-\Delta)^{-\frac{s+\alpha_1}{2}}F$ we infer the estimate
\begin{equation*}
\left\|\int_{s<t} e^{i(t-s)\Heps} F(s)\dd s\right\|_{L^p(0,T;W^{-s-\alpha_1,q}(\R^d))}\leq C {\eps}^{\alpha_1} 
\| F \|_{L^{p_1'}(0,T;{W}^{-s,q_1'}(\R^d))}
\end{equation*}
for $s\in \R$ and $\alpha_1\in [0,\frac{d-2}{2}(1-\frac1q-\frac{1}{q_1})]$. This completes the proof.
\end{proof}


\bibliographystyle{siam}
\bibliography{M125295}

\begin{thebibliography}{10}

\bibitem{AF03}
{\sc R.~A. Adams and J.~J.~F. Fournier}, {\em Sobolev spaces}, vol.~140 of Pure
  and Applied Mathematics (Amsterdam), Elsevier/Academic Press, Amsterdam,
  second~ed., 2003.

\bibitem{A}
{\sc T.~Alazard}, {\em A minicourse on the low {M}ach number limit}, Discrete
  Contin. Dyn. Syst. Ser. S, 1 (2008), pp.~365--404.

\bibitem{AHM18}
{\sc P.~{Antonelli}, L.~E. {Hientzsch}, and P.~{Marcati}}, {\em On the {C}auchy
  problem for the {QHD} system with infinite mass and energy: applications to
  quantum vortex dynamics}, in preparation.

\bibitem{AHMZ}
{\sc P.~Antonelli, L.~E. Hientzsch, P.~Marcati, and H.~Zheng}, {\em On some
  results for quantum hydrodynamical models}, in Mathematical Analysis in Fluid
  and Gas Dynamics, T.~Kobayashi, ed., vol.~2070, RIMS K\^oky\^uroku, 2018,
  pp.~107--129.

\bibitem{AHS20}
{\sc P.~Antonelli, L.~E. Hientzsch, and S.~Spirito}, {\em Global existence of
  finite energy weak solutions to the quantum {N}avier-{S}tokes equations with
  non-trivial far-field behavior}, arXiv e-prints 2001.01652,  (2020).

\bibitem{AM1}
{\sc P.~Antonelli and P.~Marcati}, {\em On the finite energy weak solutions to
  a system in quantum fluid dynamics}, Comm. Math. Phys., 287 (2009),
  pp.~657--686.

\bibitem{AM2}
\leavevmode\vrule height 2pt depth -1.6pt width 23pt, {\em The quantum
  hydrodynamics system in two space dimensions}, Arch. Ration. Mech. Anal., 203
  (2012), pp.~499--527.

\bibitem{AM16}
{\sc P.~{Antonelli} and P.~{Marcati}}, {\em Some results on systems for quantum
  fluids}, in Recent advances in partial differential equations and
  applications, vol.~666 of Contemp. Math., Amer. Math. Soc., Providence, RI,
  2016, pp.~41--54.

\bibitem{AS}
{\sc P.~Antonelli and S.~Spirito}, {\em Global existence of finite energy weak
  solutions of quantum {N}avier-{S}tokes equations}, Arch. Ration. Mech. Anal.,
  225 (2017), pp.~1161--1199.

\bibitem{AS15}
\leavevmode\vrule height 2pt depth -1.6pt width 23pt, {\em On the compactness
  of finite energy weak solutions to the quantum {N}avier-{S}tokes equations},
  J. Hyperbolic Differ. Equ., 15 (2018), pp.~133--147.

\bibitem{AS19}
\leavevmode\vrule height 2pt depth -1.6pt width 23pt, {\em {Global existence of
  weak solutions to the {N}avier-{S}tokes-{K}orteweg equations}}, arXiv
  e-prints 1903.02441,  (2019).

\bibitem{AS18}
\leavevmode\vrule height 2pt depth -1.6pt width 23pt, {\em On the compactness
  of weak solutions to the {N}avier-{S}tokes-{K}orteweg equations for capillary
  fluids}, Nonlinear Anal., 187 (2019), pp.~110--124.

\bibitem{BL}
{\sc J.~Bergh and J.~L\"{o}fstr\"{o}m}, {\em Interpolation spaces. {A}n
  introduction}, Springer-Verlag, Berlin-New York, 1976.
\newblock Grundlehren der Mathematischen Wissenschaften, No. 223.

\bibitem{BDS10}
{\sc F.~B\'{e}thuel, R.~Danchin, and D.~Smets}, {\em On the linear wave regime
  of the {G}ross-{P}itaevskii equation}, J. Anal. Math., 110 (2010),
  pp.~297--338.

\bibitem{BBCS}
{\sc C.~Boccato, C.~Brennecke, S.~Cenatiempo, and B.~Schlein}, {\em
  {B}ogoliubov theory in the {G}ross-{P}itaevskii limit}, Acta Math., 222
  (2019), pp.~219--335.

\bibitem{B}
{\sc N.~N. Bogolyubov}, {\em {On the theory of superfluidity}}, J. Phys.(USSR),
  11 (1947), pp.~23--32.
\newblock [Izv. Akad. Nauk Ser. Fiz.11,77(1947)].

\bibitem{BD}
{\sc D.~Bresch and B.~Desjardins}, {\em Quelques mod{\`e}les diffusifs
  capillaires de type {K}orteweg}, C. R. Acad. Sci. Paris, section
  m{\'e}canique, 332 (2004), pp.~881 -- 886.

\bibitem{BDL}
{\sc D.~Bresch, B.~Desjardins, and C.-K. Lin}, {\em On some compressible fluid
  models: {K}orteweg, lubrication, and shallow water systems}, Comm. Partial
  Differential Equations, 28 (2003), pp.~843--868.

\bibitem{BM}
{\sc S.~Brull and F.~M\'{e}hats}, {\em Derivation of viscous correction terms
  for the isothermal quantum {E}uler model}, ZAMM Z. Angew. Math. Mech., 90
  (2010), pp.~219--230.

\bibitem{CKH}
{\sc R.~Carles, K.~Carrapatoso, and M.~Hillairet}, {\em Global weak solutions
  for quantum isothermal fluids}, 2019.

\bibitem{COX}
{\sc Y.~Cho, T.~Ozawa, and S.~Xia}, {\em Remarks on some dispersive estimates},
  Commun. Pure Appl. Anal., 10 (2011), pp.~1121--1128.

\bibitem{DG99}
{\sc B.~Desjardins and E.~Grenier}, {\em Low {M}ach number limit of viscous
  compressible flows in the whole space}, R. Soc. Lond. Proc. Ser. A Math.
  Phys. Eng. Sci., 455 (1999), pp.~2271--2279.

\bibitem{DM08}
{\sc D.~Donatelli and P.~Marcati}, {\em A quasineutral type limit for the
  {N}avier-{S}tokes-{P}oisson system with large data}, Nonlinearity, 21 (2008),
  pp.~135--148.

\bibitem{DM15}
{\sc D.~Donatelli and P.~Marcati}, {\em Quasi-neutral limit, dispersion, and
  oscillations for {K}orteweg-type fluids}, SIAM J. Math. Anal., 47 (2015),
  pp.~2265--2282.

\bibitem{DM16}
\leavevmode\vrule height 2pt depth -1.6pt width 23pt, {\em Low {M}ach number
  limit for the quantum hydrodynamics system}, Res. Math. Sci., 3 (2016),
  pp.~Paper No. 13, 13.

\bibitem{FN}
{\sc E.~Feireisl and A.~Novotn\'{y}}, {\em Singular limits in thermodynamics of
  viscous fluids}, Advances in Mathematical Fluid Mechanics,
  Birkh\"{a}user/Springer, Cham, 2017.
\newblock Second edition of [ MR2499296].

\bibitem{F}
{\sc E.~Feireisl, A.~Novotn\'{y}, and H.~Petzeltov\'{a}}, {\em On the existence
  of globally defined weak solutions to the {N}avier-{S}tokes equations}, J.
  Math. Fluid Mech., 3 (2001), pp.~358--392.

\bibitem{Gar}
{\sc C.~L. Gardner}, {\em The quantum hydrodynamic model for semiconductor
  devices}, SIAM J. Appl. Math., 54 (1994), pp.~409--427.

\bibitem{GHN17}
{\sc Z.~Guo, Z.~Hani, and K.~Nakanishi}, {\em Scattering for the 3{D}
  {G}ross-{P}itaevskii equation}, Comm. Math. Phys., 359 (2018), pp.~265--295.

\bibitem{GPW08}
{\sc Z.~Guo, L.~Peng, and B.~Wang}, {\em Decay estimates for a class of wave
  equations}, J. Funct. Anal., 254 (2008), pp.~1642--1660.

\bibitem{GNT05}
{\sc S.~Gustafson, K.~Nakanishi, and T.-P. Tsai}, {\em Scattering for the
  {G}ross-{P}itaevskii equation}, Math. Res. Lett., 13 (2006), pp.~273--285.

\bibitem{GNT07}
\leavevmode\vrule height 2pt depth -1.6pt width 23pt, {\em Global dispersive
  solutions for the {G}ross-{P}itaevskii equation in two and three dimensions},
  Ann. Henri Poincar\'{e}, 8 (2007), pp.~1303--1331.

\bibitem{GNT09}
\leavevmode\vrule height 2pt depth -1.6pt width 23pt, {\em Scattering theory
  for the {G}ross-{P}itaevskii equation in three dimensions}, Commun. Contemp.
  Math., 11 (2009), pp.~657--707.

\bibitem{H90}
{\sc L.~H\"{o}rmander}, {\em The analysis of linear partial differential
  operators. {I}}, Classics in Mathematics, Springer-Verlag, Berlin, 2003.
\newblock Distribution theory and Fourier analysis, Reprint of the second
  (1990) edition [Springer, Berlin; MR1065993 (91m:35001a)].

\bibitem{JqNS}
{\sc A.~J\"{u}ngel}, {\em Global weak solutions to compressible
  {N}avier-{S}tokes equations for quantum fluids}, SIAM J. Math. Anal., 42
  (2010), pp.~1025--1045.

\bibitem{J}
\leavevmode\vrule height 2pt depth -1.6pt width 23pt, {\em Dissipative quantum
  fluid models}, Riv. Math. Univ. Parma (N.S.), 3 (2012), pp.~217--290.

\bibitem{JMi}
{\sc A.~J\"{u}ngel and J.-P. Mili\v{s}i\'{c}}, {\em Full compressible
  {N}avier-{S}tokes equations for quantum fluids: derivation and numerical
  solution}, Kinet. Relat. Models, 4 (2011), pp.~785--807.

\bibitem{KT98}
{\sc M.~Keel and T.~Tao}, {\em Endpoint {S}trichartz estimates}, Amer. J.
  Math., 120 (1998), pp.~955--980.

\bibitem{Khal}
{\sc I.~M. Khalatnikov}, {\em An introduction to the theory of superfluidity},
  Advanced Book Classics, Addison-Wesley Publishing Company, Advanced Book
  Program, Redwood City, CA, 1989.
\newblock Translated from the Russian by Pierre C. Hohenberg, Translation
  edited and with a foreword by David Pines, Reprint of the 1965 edition.

\bibitem{KL18}
{\sc Y.-S. Kwon and F.~Li}, {\em Incompressible limit of the degenerate quantum
  compressible {N}avier-{S}tokes equations with general initial data}, J.
  Differential Equations, 264 (2018), pp.~3253--3284.

\bibitem{LV16}
{\sc I.~Lacroix-Violet and A.~Vasseur}, {\em Global weak solutions to the
  compressible quantum {N}avier-{S}tokes equation and its semi-classical
  limit}, J. Math. Pures Appl. (9), 114 (2018), pp.~191--210.

\bibitem{L34}
{\sc J.~Leray}, {\em Sur le mouvement d'un liquide visqueux emplissant
  l'espace}, Acta Math., 63 (1934), pp.~193--248.

\bibitem{LX}
{\sc J.~Li and Z.~Xin}, {\em Global existence of weak solutions to the
  barotropic compressible {N}avier-{S}tokes flows with degenerate viscosities},
  2015.

\bibitem{L96}
{\sc P.-L. Lions}, {\em Mathematical topics in fluid mechanics. {V}ol. 2},
  vol.~10 of Oxford Lecture Series in Mathematics and its Applications, The
  Clarendon Press, Oxford University Press, New York, 1998.
\newblock Compressible models, Oxford Science Publications.

\bibitem{LM98}
{\sc P.-L. Lions and N.~Masmoudi}, {\em Incompressible limit for a viscous
  compressible fluid}, J. Math. Pures Appl. (9), 77 (1998), pp.~585--627.

\bibitem{M}
{\sc N.~Masmoudi}, {\em Examples of singular limits in hydrodynamics}, in
  Handbook of differential equations: evolutionary equations. {V}ol. {III},
  Handb. Differ. Equ., Elsevier/North-Holland, Amsterdam, 2007, pp.~195--275.

\bibitem{NNT}
{\sc Q.-H. Nguyen, P.-T. Nguyen, and B.~Q. Tang}, {\em Energy equalities for
  compressible {N}avier-{S}tokes equations}, Nonlinearity, 32 (2019),
  pp.~4206--4231.

\bibitem{PS}
{\sc L.~Pitaevskii and S.~Stringari}, {\em Bose-{E}instein condensation and
  superfluidity}, vol.~164, Oxford University Press, 2016.

\bibitem{S}
{\sc B.~Schlein}, {\em Bogoliubov excitation spectrum for {B}ose-{E}instein
  condensates}, in Proceedings of the {I}nternational {C}ongress of
  {M}athematicians---{R}io de {J}aneiro 2018. {V}ol. {III}. {I}nvited lectures,
  World Sci. Publ., Hackensack, NJ, 2018, pp.~2669--2686.

\bibitem{U}
{\sc S.~Ukai}, {\em The incompressible limit and the initial layer of the
  compressible {E}uler equation}, J. Math. Kyoto Univ., 26 (1986),
  pp.~323--331.

\bibitem{YY}
{\sc J.~Yang, Q.~Ju, and Y.-F. Yang}, {\em Asymptotic limits of
  {N}avier-{S}tokes equations with quantum effects}, Z. Angew. Math. Phys., 66
  (2015), pp.~2271--2283.

\end{thebibliography}

\end{document}